\documentclass[british]{amsart}
\usepackage{stmaryrd} 
\usepackage{amssymb}
\usepackage{enumitem}
\usepackage{soul}
\usepackage{mathabx} 
\usepackage{amsmath} 
\usepackage{amscd}
\usepackage{amsbsy}
\usepackage{cancel}
\usepackage{comment, enumitem}
\usepackage[all]{xy}
\usepackage{hyperref}
\hypersetup{
 colorlinks=true,
 linkcolor=violet,
 filecolor=blue,
 citecolor=brown,
 urlcolor=cyan,
 pdftitle={FSG: Conjectures A and B},
 }
\usepackage{mathrsfs}
\usepackage{color}
\usepackage{mathtools,caption}
\usepackage{tikz-cd}
\usepackage{longtable}
\usepackage[utf8]{inputenc}
\usepackage[OT2,T1]{fontenc}
\usepackage[normalem]{ulem}
\usepackage{commath}
\usepackage{geometry}

\DeclareMathOperator{\nr}{nr}
\DeclareMathOperator{\Frac}{Frac}

\DeclareMathOperator{\cs}{cs}
\DeclareMathOperator{\ram}{ram}

\DeclareMathOperator{\Ann}{Ann}
\DeclareMathOperator{\Hom}{Hom}

\DeclareMathOperator{\Ind}{Ind}
\DeclareMathOperator{\Gal}{Gal}
\DeclareMathOperator{\rank}{rank}
\DeclareMathOperator{\ord}{ord}
\DeclareMathOperator{\Cl}{Cl}
\DeclareMathOperator{\Pic}{Pic}

\DeclareMathOperator{\Sel}{Sel}
\DeclareMathOperator{\coker}{coker}
\DeclareMathOperator{\GL}{GL}
\DeclareMathOperator{\cyc}{cyc}
\DeclareMathOperator{\ac}{ac}
\DeclareMathOperator{\bad}{bad}
\DeclareMathOperator{\cont}{cont}
\DeclareMathOperator{\cd}{cd}
\DeclareMathOperator{\height}{ht}

\newcommand{\scrF}{\mathscr{F}}
\newcommand{\QQ}{\mathbb{Q}}
\newcommand{\ZZ}{\mathbb{Z}}
\newcommand{\Z}{\mathcal{Z}}
\newcommand{\OK}{\mathcal{O}_K}
\newcommand{\X}{\mathfrak{X}}

\newcommand{\Y}{\mathfrak{Y}}
\renewcommand{\lg}[1][G]{\Lambda(#1)}
\newcommand{\lgK}[1][G]{\widetilde{\Lambda(#1)}}
\newcommand{\qh}{\mathcal{Q}(H)}
\newcommand{\qhK}{\widetilde{\qh}}
\newcommand{\linf}{\mathcal{L}}
\newcommand{\Zp}{\ZZ_p}
\newcommand*{\EC}{\mathsf{E}}

\DeclareSymbolFont{cyrletters}{OT2}{wncyr}{m}{n}
\DeclareMathSymbol{\Sha}{\mathalpha}{cyrletters}{"58}

\theoremstyle{plain}
\newtheorem*{Reform*}{Reformulation}
\newtheorem{Th}{Theorem}[section]
\newtheorem{Lemma}[Th]{Lemma}
\newtheorem{prop}[Th]{Proposition}
\newtheorem{cor}[Th]{Corollary}
\makeatletter
\theoremstyle{plain}
\newtheorem*{intr@thm}{\intr@thmname}
\newenvironment{introtheorem}[1]{%
	\def\intr@thmname{Theorem #1}
	\begin{intr@thm}}
	{\end{intr@thm}}
\newtheorem*{c@njecture}{\conjn@name}
\newcommand{\myl@bel}[2]{%
 \protected@write \@auxout {}{\string \newlabel {#1}{{#2}{\thepage}{#2}{#1}{}} }%
 \hypertarget{#1}{}
  } 
\newenvironment{labelledconj}[3][]%
  {
    \def\conjn@name{#2}
    \begin{c@njecture}[{#1}]\myl@bel{#3}{#2}
  }
  {
    \end{c@njecture}
  }
\makeatother
\theoremstyle{definition}
\newtheorem{Defi}[Th]{Definition}
\theoremstyle{remark}
\newtheorem{Remark}[Th]{Remark}
\newtheorem*{Remark*}{Remark}
\newsavebox{\fieldsdiagram}
\begin{lrbox}{\fieldsdiagram}
\begin{tikzpicture}[node distance = 1.75cm, auto]
 \node (Q) {$\mathbb{Q}$};
 \node (K) [node distance=.7cm, above of=Q] {$K$};
 \node (tildeK) [node distance=2.5cm, above of=K, right of=K] {$\widetilde{K}$};
 \node (Kprime) [node distance=1cm, above of=K, left of=K] {$K'$};
 \node (F)[node distance=1cm, above of=Kprime, left of=Kprime] {$F$};
 \node (Fcyc) [node distance=1.25cm, above of=F, right of=F] {$F_{\cyc}$}; 
 \node (Finf) [node distance=1.25cm, above of=Fcyc, right of=Fcyc] {$F_\infty$};
 \node (Kprimeinf) [node distance=2.5cm, above of=Kprime, right of=Kprime] {$K'_\infty$};
 \node (MFinf)[node distance=1cm, above of=Finf] {$M(F_\infty)$};
 \node (tildeF) [node distance=3.5cm, above of=F] {$\widetilde{F}$};
 \node (MtildeF)[node distance=1cm, above of=tildeF] {$M(\widetilde{F})$};
 \node (tildeFinf) [node distance=1cm, above of=MFinf] {${L(F_\infty)}$};
 \node (tildeL) [node distance=1cm, above of=MtildeF] {$L(\widetilde{F})$}; 
 \node (tildeT) [node distance=1cm, above of=tildeL] {$\scrF_S$}; 
 \draw[-] (Q) to node {} (K);
 \draw[-] (K) to node {\tiny$\Zp^2$} (tildeK);
 \draw[-] (tildeK) to node {} (Kprimeinf);
 \draw[-] (Finf) to node {} (Kprimeinf);
 \draw[-] (F) to node {\tiny$\Zp$} (Fcyc);
 \draw[-] (Fcyc) to node {\tiny$\Zp$} (Finf);
 \draw[-] (Kprime) to node {\tiny$\Delta'$} (F);
 \draw[-] (Kprime) to node {} (Kprimeinf);
 \draw[-] (K) to node {\tiny$\Cl(K)$} (Kprime);
 \draw[-] (Finf) to node [swap]{} (MFinf);
 \draw[-] (tildeF) to node {} (MtildeF);
 \draw[-] (Finf) to node {} (tildeF); 
 \draw[-] (MFinf) to node {} (tildeFinf); 
 \draw[-] (tildeFinf) to node {} (tildeL);
 \draw[-] (tildeL) to node {} (tildeT);
 \draw[-] (MtildeF) to node {} (tildeL);
 \draw[-] (F) to node {} (tildeF); 
 \draw[-] (Fcyc) to node {} (tildeF);
 \draw[bend left=70] (K) to node {\tiny$\Delta$} (F);
 \draw[bend right] (F) to node [swap]{\tiny$G$} (Finf);
 \draw[bend right=55] (Finf) to node [swap]{\tiny$X_{\nr}^{F_\infty}$} (tildeFinf);
 \draw[bend left=75] (tildeF) to node {\tiny$X_{S}^{\widetilde{F}}$} (tildeT);
 \draw[bend right=55] (tildeF) to node [swap]{\tiny$X_{\nr}^{\widetilde{F}}$} (tildeL);
\end{tikzpicture}
\end{lrbox}

\begin{document}
\title[Fine Selmer Groups in abelian $p$-adic Lie extensions]{Structure of Fine Selmer Groups \\ in abelian $p$-adic Lie extensions}
\author[D.~Kundu]{Debanjana Kundu}
\address[Kundu]{Fields Institute \\ University of Toronto \\
 Toronto ON, M5T 3J1, Canada}
\email{dkundu@math.toronto.edu}

\author[F.~A.~E.~Nuccio]{Filippo A.~E.~Nuccio Mortarino Majno di Capriglio}
\address[Nuccio Mortarino Majno di Capriglio]{Université Jean Monnet Saint-Étienne, CNRS, Institut Camille Jordan UMR 5208, \mbox{F-42023}, Saint-Étienne, France}
\email{filippo.nuccio@univ-st-etienne.fr}

\author[R.~Sujatha]{Sujatha Ramdorai}
\address[Sujatha]{Mathematics Department, 
1984, Mathematics Road,
University of British Columbia, 
Vancouver, Canada, V6T1Z2}
\email{sujatha@math.ubc.ca}

\date{April 21, 2023}

\keywords{Fine Selmer groups, pseudonull, Conjecture A, Conjecture B, Generalized Greenberg's Conjecture}
\subjclass[2020]{Primary 11R23, 11G05; Secondary 11R34}

\begin{abstract}
This paper studies fine Selmer groups of elliptic curves in abelian $p$-adic Lie extensions.
A class of elliptic curves are provided where both the Selmer group and the fine Selmer group are trivial in the cyclotomic $\Zp$-extension.
The fine Selmer groups of elliptic curves with complex multiplication are shown to be pseudonull over the trivializing extension in some new cases.
Finally, a relationship between the structure of the fine Selmer group for some CM elliptic curves and the Generalized Greenberg's Conjecture is clarified.
\end{abstract}

\maketitle

\section{Introduction} \label{section: intro}
The fine Selmer group (see \S\ref{fine selmer definition}) is a module over an Iwasawa algebra that is of interest in the arithmetic of elliptic curves.
It plays a key role in the formulation of the main conjecture in Iwasawa theory.
Moreover, it enables us to propose analogues of important conjectures in classical Iwasawa theory to elliptic curves over certain $p$-adic Lie extensions of their field of definition.
J.~Coates and the third named author initiated a systematic study of the structure of fine Selmer groups and proposed two conjectures (see~\cite[Conjectures A and B]{CS05}).
While \ref{conj:A} is a generalization of the \ref{conj:Iwasawa} to the context of elliptic curves, \ref{conj:B} is in the spirit of generalizing R.~Greenberg's pseudonullity conjecture to elliptic curves.
Recently, there has been a renewed interest in studying pseudonull modules over Iwasawa algebras,~\cite{BCG+15, LP19_pseudonull}.
It is thus natural to investigate \ref{conj:B}, and this article makes progress in this direction.
These conjectures have been generalized to fine Selmer groups of ordinary Galois representations associated to modular forms in \cite{SJ11}, and their $\mod{p}$-versions for supersingular elliptic curves have been studied by the second and third author in \cite{NucSuj21}.
This article restricts attention to the fine Selmer groups of elliptic curves, with good reduction at a prime $p$, over abelian $p$-adic Lie extensions of the base field.

We now outline the main results in the paper.
Given a number field $F$ and an odd prime number $p$, let $\EC/F$ be an elliptic curve, with good reduction at all the primes of $F$ that lie above $p$.
Consider an admissible $p$-adic Lie extension ${\linf}$ of $F$ (see \S\ref{prelim: number field} for the precise definition) with Galois group $\Gal({\linf}/F)=:{G_{\linf/F}}$.
The dual fine Selmer group of $\EC$ at a prime $p$ over ${\linf}$ is a finitely generated module over the associated Iwasawa algebra (see \S\ref{fine selmer definition}).
While \ref{conj:A} asserts that the dual fine Selmer group over the cyclotomic $\Zp$-extension $F_{\cyc}/F$ is finitely generated as a $\Zp$-module,~\ref{conj:B} is an assertion on the structure of the dual fine Selmer group over admissible $p$-adic Lie extensions of dimension at least 2.
This conjecture predicts that the dual fine Selmer group over any admissible $p$-adic Lie extension is pseudonull as a module over the associated Iwasawa algebra.
In this article, both conjectures are established in previously unknown cases.
Using a result of Greenberg, we prove a general theorem that gives sufficient conditions for the dual fine Selmer group of $\EC$ over the cyclotomic $\Zp$-extension $F_{\cyc}$ to be trivial.
More precisely, we have the following theorem (we refer the reader to Corollary~\ref{cor: conj A for CM} for finer estimates):
\begin{introtheorem}{\ref{Theorem: Construction}}
Let $\EC/F$ be the base-change of a rational elliptic curve $\EC/\QQ$.
Suppose that it has rank $0$ over $F$ and that the Shafarevich--Tate group of $\EC/F$ is finite.
When $\EC$ has CM by an order of an imaginary quadratic field $K$, assume further that the Galois closure of $F$, denoted by $F^c$, contains $K$.
Then, the Selmer group $\Sel\left(\EC/F_{\cyc}\right)$ is trivial for a set of prime numbers of density at least $\frac{1}{[F^c:\QQ]}$.
In particular, \ref{conj:A} holds for $\EC/F$ at all such primes.
\end{introtheorem}

Denote by $F(\EC_{p^\infty})$ the field obtained by adjoining the coordinates of all $p$-power torsion points.
When $p$ is a prime of good ordinary reduction, using a result of B.~Perrin-Riou~\cite[Lemme~1.1(\romannumeral 1) and Lemme~1.3]{PR81} we prove that \ref{conj:B} holds for special classes of admissible $p$-adic Lie extensions whenever the dual fine Selmer group over the cyclotomic extension is finite for a CM elliptic curve.
We obtain the following result:
\begin{introtheorem}{\ref{theorem: main theorem CM}}
Let $\EC/F$ be an elliptic curve defined over a number field $F$.
Suppose that $F$ contains the imaginary quadratic field $K$ and that $\EC$ has CM by $\OK$.
Assume further that $p\geq 3$ is a prime of good ordinary reduction that splits in $K$
and that $\Gal(F(\EC_{p^\infty})/F)\simeq \Zp^2$.
If the fine Selmer group over the cyclotomic $\Zp$-extension $F_{\cyc}/F$ is finite, then \ref{conj:B} holds for $(\EC,F(\EC_{p^\infty}))$.
\end{introtheorem}

Over the cyclotomic $\Zp$-extension $F_{\cyc}$ of $F$, there is a connection between the Galois group of the maximal abelian unramified pro-$p$ extension of $F_{\cyc}$ and the fine Selmer groups of elliptic curves defined over $F$, see~\cite[Theorem~3.4]{CS05}.
This phenomenon can be extended to (both abelian and non-abelian) admissible $p$-adic Lie extensions of higher dimension.
In fact, \ref{conj:B} can be viewed as an elliptic curve analogue of an old conjecture of Greenberg on Galois modules associated with pro-$p$ Hilbert class fields (see~\S\ref{section: Conjectures} for the precise statement).
This has been explored in~\cite[p.~827]{CS05}.
It is therefore pertinent to investigate the precise connections between \ref{conj:B} for admissible, abelian $p$-adic Lie extensions, and Greenberg's conjecture.
For CM elliptic curves, the \ref{conj:GGC} is shown to be equivalent to \ref{conj:B} for certain admissible pro-$p$, $p$-adic Lie extensions in Theorem~\ref{thm:appendix}.
This result provides a framework for proving new cases of the~\ref{conj:GGC}.
In particular, we prove the following result\footnote{
The proof of Theorem~\ref{Thm: conj B implies GGC in several} does not require $\EC$ to be defined over $K$.
This formulation is used in this introduction, for simplicity.}.

\begin{introtheorem}{\ref{Thm: conj B implies GGC in several} and Corollary~\ref{cor for GGC of K'Ep}}
Let $K/\QQ$ be an imaginary quadratic field.
If there exists one CM elliptic curve $\EC/K$ such that the dual fine Selmer group is pseudonull over the trivializing extension $K(\EC_{p^\infty})$, then the \ref{conj:GGC} holds for $K$ and $K(\EC_p)$.
\end{introtheorem}

Little is known about~\ref{conj:B} and the~\ref{conj:GGC} in full generality.
We recall some cases where~\ref{conj:B} is proven in the literature.
When there is a unique prime above $p$ in the $p$-adic Lie extension of interest~\ref{conj:B} is proven in \cite[Theorem~1.3]{Och09} and \cite[\S~4]{She18}.
Also, when the $p$-adic Lie extension has large dimension there are explicit examples where~\ref{conj:B} is known, detailed in~\cite[Example~23]{Bha07}.
Certain analogues of \ref{conj:B} have also been considered in \cite{Jha12, LP19_pseudonull}.
For evidence towards the~\ref{conj:GGC} (both theoretical and computational) see~\cite[Remark~1.3]{Tak20}, as well as \cite{Min86, McC01, Oza01, DV05, Sha08, Fuj17}.
As per the knowledge of the authors, most results in this latter direction require the crucial hypothesis that $p$ does not divide the class number of the number field.
One exception is the result of R.~Sharifi and W.~McCallum, where the conjecture for $\QQ(\mu_p)$ is proven under certain assumptions on a cup-product (see \cite[Corollary~10.5]{MS03}); another is of Sharifi \cite[Theorem~1.3]{Sha08}, where computational evidence for the \ref{conj:GGC} is provided when $F=\QQ(\mu_p)$ and $p<1000$ is an irregular prime.
Our approach suggests a new line of attack for the \ref{conj:GGC} even in the case when $p$ divides the class number of the base field.

The paper consists of five sections.
Section~\ref{section:Preliminaries} is preliminary in nature; wherein we recall the precise assertions of \ref{conj:A}, \ref{conj:B}, and the~\ref{conj:GGC} and we introduce the main objects of study.
In Section~\ref{section:Finite Fine Selmer Groups in the Cyclotomic Extension}, new evidence for \ref{conj:A} is provided by proving the triviality of the fine Selmer group over the cyclotomic extension.
Some simple cases of \ref{conj:B} are proven in Section~\ref{section: Special Cases}.
In Section~\ref{section actually relating ggc to conj B} the relation between \ref{conj:B} for CM elliptic curves and the~\ref{conj:GGC} is clarified.

\section{Preliminaries} \label{section:Preliminaries}

Throughout this article, $p$ denotes an odd prime number.
For an abelian group $M$ and a positive integer $n$, 
write $M_{p^n}$ for the subgroup of elements of $M$ annihilated by $p^n$.
Put
\[
M_{p^\infty} :=\underset{n\geq 1}{\bigcup} \quad M_{p^n}, \quad T_p(M):=\varprojlim M_{p^n}.
\]
and, when $M$ is a discrete $p$-primary (\emph{resp.}~compact pro-$p$) abelian group $M$, its Pontryagin dual is defined as
\[
M^\vee = \Hom_{\cont}(M,\QQ_p/\Zp).
\]
Given any $p$-adic analytic group $G$, its \emph{Iwasawa algebra} is defined as
\[
\lg=\varprojlim _{U}\Zp[G/U]
\]
for $U$ running through all open, normal subgroups of $G$.
When $G$ is compact and $p$-valued in the sense of M.~Lazard, $\lg$ is a noetherian Auslander regular ring (see~\cite[Proposition~6.2]{CoaSchSuj03}).
In the special case when $G$ is abelian with no elements of order $p$, there is an isomorphism
\[
\lg\simeq \Zp\llbracket T_1, \ldots, T_d\rrbracket.
\]
where $d$ is the dimension of $G$ as a $p$-adic analytic manifold.
If $M$ is a compact (\emph{resp.}~discrete) $\lg$-module then its Pontryagin dual is discrete (\emph{resp.}~compact).
Given a finitely generated $\lg$-module $M$, its Krull dimension is defined as the Krull dimension of $\lg/\Ann(M)$ and it is denoted $\dim(M)$.

\subsection{}
Suppose that $G$ is an abelian $p$-analytic group without elements of order $p$.
A finitely generated $\lg$-module $M$ is \emph{torsion} (\emph{resp.}~\emph{pseudonull}) if $\dim(M)\leq \dim\left(\lg \right)-1$ (\emph{resp.}~$\dim(M) \leq \dim\left(\lg\right)-2$).
Equivalently (see \cite[p.~273]{Ven02}), $M$ is pseudonull if there exists a prime ideal $\mathfrak{p}$ such that
\[
\Ann_{\lg}(M):=\{a\in\lg: aM=0\}\subseteq \mathfrak{p}
\]
and $\height(\mathfrak{p})\geq 2$ (see \cite[Definition~5.1.4]{NSW08}).

Let $W$ (\emph{resp.}~$M$) be a discrete (\emph{resp.}~compact) $G$-module.
The profinite cohomology groups (\emph{resp.}~homology groups) of $W$ (\emph{resp.}~$M$) are denoted $H^i(G,W)$ (\emph{resp.}~$H_i(G,M)$).
The subgroup of elements of $W$ fixed by $G$ is denoted $W^G$, and $M_G$ denotes the largest quotient of $M$ on which $G$ acts trivially.

\subsection{}
\label{prelim: number field}
For a number field $F$, denote by $F_{\cyc}$ its cyclotomic $\Zp$-extension.
Suppose that $S=S(F)$ is a finite set of primes of $F$ containing the primes above $p$ and the archimedean primes.
Let $F_S$ be the maximal extension of $F$ unramified outside $S$ and set $G_S(F)=\Gal(F_S/F)$.
For any (finite or infinite) extension $\linf/F$ contained in $F_S$, denote by $G_S(\linf)$ the Galois group $\Gal(F_S/\linf)$.
Throughout the paper, the focus is on $S$-admissible $p$-adic Lie extensions $\linf/F$, in the following sense:
\begin{Defi}
An \emph{$S$-admissible} $p$-adic Lie extension is a Galois extension $\linf/F$ satisfying the following conditions:
\begin{itemize}
\item the group $\Gal\left({\linf}/F\right)$ is a pro-$p$, $p$-adic Lie group with no elements of order $p$;
\item the field ${\linf}$ contains the cyclotomic $\Zp$-extension $F_{\cyc}$;
\item the field ${\linf}$ is contained in $F_S$.
\end{itemize}
\end{Defi}

Next, we record some conjectures pertaining to the modules associated with maximal abelian unramified pro-$p$ extension of admissible $p$-adic Lie extensions.
The first conjecture we mention was formulated by K.~Iwasawa in \cite[pp.~1--2]{Iwa73_AC} for the cyclotomic $\Zp$-extension.

\begin{labelledconj}{Iwasawa $\mu=0$ Conjecture}{conj:Iwasawa}
Let $L(F_{\cyc})$ denote the maximal abelian unramified pro-$p$ extension of $F_{\cyc}$ and set
\[
X_{\nr}^{F_{\cyc}} = \Gal(L(F_{\cyc})/F_{\cyc}).
\]
Then, the $\mu$-invariant associated with $X_{\nr}^{F_{\cyc}}$ is trivial.
\end{labelledconj}

In \cite[Theorem~5]{Iwa73}, Iwasawa proved that $X_{\nr}^{F_{\cyc}}$ is a torsion $\lg[\Gamma]$-module; in view of this result, the~\ref{conj:Iwasawa} is equivalent to saying that $X_{\nr}^{F_{\cyc}}$ is finitely generated over $\Zp$.
When $F/\QQ$ is an abelian extension, the \ref{conj:Iwasawa} is known to be true by the work \cite{FW79} by B.~Ferrero and L.~Washington.

Next, we mention a conjecture of Greenberg (see~\cite[Conjecture~3.5]{Gre01}) which is formulated for certain abelian $p$-adic Lie extensions.

\begin{labelledconj}{Generalized Greenberg's Conjecture}{conj:GGC}
Let $\widetilde{F}$ denote the compositum of all $\Zp$-extensions of $F$ and let $L(\widetilde{F})$ denote the maximal abelian unramified pro-$p$ extension of $\widetilde{F}$.
Then $\Gal \left(L(\widetilde{F})/ \widetilde{F}\right)$ is a pseudonull module over the Iwasawa algebra $\lg[\Gal (\widetilde{F}/F)]=\Zp\llbracket \Gal (\widetilde{F}/F) \rrbracket $.
\end{labelledconj}

\subsection{}
\label{fine selmer definition}
Fix a number field $F$ and an admissible extension $\linf/F$.
Write $G_{\linf/F}$ for the compact, pro-$p$, $p$-adic Lie group $\Gal(\linf/F)$ and $\lg[G_{\linf/F}]$ for the associated Iwasawa algebra.
The main objects of study will be modules over $\lg[G_{\linf/F}]$ that arise in Iwasawa theory, such as the Selmer group and the fine Selmer group.
Let $\EC$ be an elliptic curve defined over $F$.
Choose a set $S=S(F)$ containing the primes above $p$, the primes of bad reduction of $\EC/F$, and the archimedean primes.
Write $S\supseteq S_p \cup S_{\bad} \cup S_{\infty}$, where the notation $S_p$, $S_{\bad}$, and $S_{\infty}$ are self-explanatory.
For a finite extension $L/F$ and a prime $v$ of $F$, define 
\begin{equation}
\label{defi of J and K}
 J_v(L) = \bigoplus_{w\mid v} H^1\left(L_w, \EC\right)(p), \qquad \text{and } \qquad
 K_v(L) = \bigoplus_{w\mid v} H^1\left(L_w, \EC_{p^\infty}\right)
\end{equation}
where the direct sum is taken over all primes $w$ of $L$ lying above $v$.
Taking direct limits, define
\begin{align*}
 J_v(\linf) = \varinjlim_{L}J_v(L), \qquad \text{ and } \qquad
 K_v(\linf) = \varinjlim_{L}K_v(L)
\end{align*}
where $L$ varies over finite sub-extensions of $\linf/F$.
Given any finite extension $L/F$ contained in $\linf$, the \emph{$p$-primary Selmer group} $\Sel\left(\EC/L\right)$ and the \emph{$p$-primary fine Selmer group} $R\left(\EC/L\right)$ are defined by the exactness of the following sequences:
\begin{align*}
\begin{matrix}
0&\longrightarrow &\Sel\left(\EC/L\right)&\longrightarrow &H^1\left(G_S(F), \EC_{p^\infty}\right) &\longrightarrow& \displaystyle{\bigoplus_{v\in S(L)}J_v(L)},
\cr 0&\longrightarrow &R\left(\EC/L\right) &\longrightarrow &H^1\left(G_S\left(F\right), \EC_{p^\infty}\right) &\longrightarrow & \displaystyle{\bigoplus_{v\in S(L)} K_v(L)}.
\end{matrix}
\end{align*}
Moreover, by \cite[Equation~(58)]{CS05} we can relate these groups as follows 
\begin{equation}\label{eq 58 cs05}
0 \longrightarrow R\left(\EC/L\right) \longrightarrow \Sel\left(\EC/L\right) \longrightarrow \bigoplus_{w\in S_p(L)} \left(\EC\left(L_w\right) \otimes \QQ_p/\ZZ_p \right).
\end{equation}
Define $\Sel(\EC/\linf) =\varinjlim_L \Sel(\EC/L) $ and $R(\EC/\linf) =\varinjlim_L R(\EC/L)$.
It can then be shown (see \cite[pp.~14--15]{CS00_book} and \cite[Equation~(46)]{CS05}) that
\begin{align*}
\Sel\left(\EC/\linf\right) \cong &\ker\left( H^1\left(G_S\left(\linf\right), \EC_{p^\infty}\right) \longrightarrow \bigoplus_{v\in S} J_v(\linf) \right) \intertext{ and }
R\left(\EC/\linf\right) \cong &\ker\left( H^1\left(G_S\left(\linf\right), \EC_{p^\infty}\right) \longrightarrow \bigoplus_{v\in S} K_v(\linf) \right).
\end{align*}
Taking direct limits of \eqref{eq 58 cs05}, we obtain that
\[
0 \longrightarrow R\left(\EC/\linf\right) \longrightarrow \Sel\left(\EC/\linf\right) \longrightarrow \varinjlim_L \bigoplus_{w\in S_p(L)} \left(\EC\left(L_w\right) \otimes \QQ_p/\ZZ_p \right).
\]
Finally, we set a notation for the Pontryagin dual of these groups:
\begin{equation}\label{eqn: dual selmer and fine selmer}
\X(\EC/{\linf}) :=\Sel\left(\EC/\linf\right)^\vee \quad \text{and} \quad \Y(\EC/{\linf}):= R\left(\EC/\linf\right)^\vee.
\end{equation}
These are compact $\lg[G_{\linf/F}]$-modules and it follows from~\eqref{eq 58 cs05} that $\Y(\EC/{\linf})$ is a quotient of $\X(\EC/{\linf})$.

In this paper, we are interested in a certain class of $S$-admissible $p$-adic Lie extensions generated by the $p$-primary torsion points of an elliptic curve.
When the elliptic curve $\EC/F$ is clear from the context, we write
\[
F_{\infty}:=\underset{n\geq 1}{\bigcup} \,F(\EC_{p^n}).
\]
It follows from the Weil pairing that $F_{\infty}$ contains $F_{\cyc}$ and the choice of $S$ ensures that $F_\infty$ is contained in $F_S$.
The Galois group $\Gal(F_{\infty}/F)$ has no $p$-torsion if $p\geq 5$ (see, for example, \cite[Lemma~4.7]{How_thesis}) and contains an open, normal, pro-$p$ subgroup (see \cite[Corollary~8.34]{DdSMS03}).
In fact, the extension $F_{\infty}/F(\EC_p)$ is always pro-$p$ and hence $S$-admissible.
If $\EC$ is an elliptic curve with CM, and $F$ contains the field of complex multiplication, then $\Gal(F_{\infty}/F)$ contains an open subgroup which is abelian and isomorphic to $\Zp^2$.

\subsection{}
\label{section: Conjectures}
Fix a number field $F$.
In this section, we record the two conjectures formulated by Coates and the third named author in \cite{CS05} which will be studied in this paper.

\begin{labelledconj}{Conjecture~\textup{A}}{conj:A}\textup{(}\cite[Section~3]{CS05}\textup{)}
Let $\EC$ be an elliptic curve defined over $F$.
Then $\Y(\EC/F_{\cyc})$ is a finitely generated $\Zp$-module.
\end{labelledconj}

This conjecture is closely related to the~\ref{conj:Iwasawa}.
Their connection can be made precise:

\begin{Th}[{\cite[Theorem~3.4]{CS05}}]
\label{thm: cs05 thm 3.4}
Let $\EC/F$ be an elliptic curve and suppose that $\Gal(F_\infty/F)$ is pro-$p$.
Then \ref{conj:A} for $\EC/F$ is equivalent to the~\ref{conj:Iwasawa} for $F$.
\end{Th}

In \cite[Theorems~A,C]{Ari14}, there are more examples for which~\ref{conj:A} holds.

The dimension theory for finitely generated modules over Iwasawa algebras allows framing an analogue of the \ref{conj:GGC} in a more general setting.
This is \ref{conj:B} and concerns the dual fine Selmer group over admissible $p$-adic Lie extensions (not necessarily abelian) of dimension $\geq 2$.
It asserts that this module is smaller than intuitively expected.

\begin{labelledconj}{Conjecture~\textup{B}}{conj:B}\textup{(}\cite[Section~4]{CS05}\textup{)}
Let $\EC/F$ be an elliptic curve and let $\linf/F$ be an $S$-admissible $p$-adic Lie extension such that $G_{\linf/F}=\Gal(\linf/F)$ has dimension strictly greater than 1.
Then \ref{conj:A} holds for $\EC/F$ and $\Y\left(\EC/\linf\right)$ is a pseudonull $\lg[G_{\linf/F}]$-module.
\end{labelledconj}

\subsection{}
\label{section from Kato's work}
Fix a number field $F$ and let $T$ denote a finitely generated $\Zp$-module, endowed with a continuous action of $G_S(F)$, where $S$ contains the primes above $p$, the archimedean primes, and the primes $v$ such that the inertia group of $v$ does not act trivially on $T$.
Note that if $T$ is the Tate module $T_p\EC$ of an elliptic curve $\EC/F$, then the inertia group of $v$ acts trivially on $T$ for every prime $v$ of good reduction.
Fix an $S$-admissible extension $\linf/F$.
Define the $i$-th \emph{Iwasawa cohomology group}
as the inverse limit
\begin{equation}\label{eqn: iwasawa cohomology defn}
\Z^i_S\left(T/\linf\right) = \varprojlim_L H^i\left( G_S(L), T\right), \text{ for } i=0, 1, 2, 
\end{equation}
where $L$ ranges over all finite extensions of $F$ contained in $\linf$ and the limit is taken with respect to the corestriction maps.
It is well-known that $\Z^0_S (T/\linf)$ vanishes (see, for example, \cite[Proposition~2.1]{CS05}).
In this article, we consider $T=\Zp(1) = \varprojlim \mu_{p^n}$or $T=T_p(\EC) = \varprojlim_n {\EC_{p^n}}$.
Here $\Zp(1)$ denotes the Tate twist of $\Zp$.
We remark that the dual fine Selmer group $\Zp(1)$ has also been studied under various guises in \cite{Sch79,NQD}.
The weak Leopoldt conjecture is known to be true for the cyclotomic $\Zp$-extension, see \cite[Theorem~10.3.25]{NSW08}.
In other words,
\[
H^2\left(G_S(F_{\cyc}),\QQ_p/\ZZ_p \right) =0.
\]
Hence $H^2\left(G_S(\linf),\QQ_p/\ZZ_p \right)$ vanishes (see~\cite[p.~815~(20)]{CS05}).
An argument identical to \cite[Lemma~3.1]{CS05} but for the module $\QQ_p/\ZZ_p$, shows that this vanishing is equivalent to the fact that $\Z^2_S\left(\Zp(1)/\linf\right)$ is $\lg[G_{\linf/F}]$-torsion.
Analogously, \cite[Lemma~3.1]{CS05} shows that $H^2\left(G_S(\linf), \EC_{p^\infty}\right)=0$ if and only if $\Z^2_S\left(T_p(\EC)/\linf\right)$ is $\lg[G_{\linf/F}]$-torsion, but the equivalent of the weak Leopoldt conjecture is not known in the case of elliptic curves.
When $G_S(\linf)$ acts trivially on $\EC_{p^\infty}$, then $H^2\left(G_S(\linf), \EC_{p^\infty}\right)=0$ (see for example \cite[Lemma~2.4]{CS05}).

The following notions will be useful in the reformulation of \ref{conj:B} in Section~\ref{reformulation}.
For $i\geq 0$ and $T=\Zp(1)$, choose $S$ to be a finite set of places of $F$ containing the primes above $p$ and the archimedean primes.
For a finite extension $L/F$, let $\mathcal{O}_L[1/S]$ be the subring of $L$ consisting of elements that are integral at every finite place of $L$ not lying over $S$, and let $H^i_{\text{\'et}}$ denote {\'e}tale cohomology.
An equivalent definition of the $i$-th Iwasawa cohomology group is the following (see \cite[\S~2.2 p.~552]{Kat06})
\begin{equation}\label{equivalent definition of Z2}
\Z^i_S\left( \Zp(1)/\linf\right) = \varprojlim_L H^i_{\text{\'et}}\left( \mathcal{O}_L[1/S], \Zp(1)\right)
\end{equation}
where $L$ ranges over all finite extensions of $F$ contained in $\linf$ and the limit is taken with respect to the corestriction maps.
The dual fine Selmer group of $\Zp(1)$ was introduced in \cite{CS05_CM} and is defined as $\Gal(M(\linf)/\linf)$ where $M(\linf)$ is the maximal abelian, pro-$p$ unramified extension of ${\linf}$ such that all primes above $p$ split completely. An equivalent definition has been given in \cite[\S2.4, p.~554]{Kat06}.
In particular,
\begin{equation}\label{eq:iso_selmer_X}
\Y(\Zp(1)/\linf) = \varprojlim_L \Pic\left(\mathcal{O}_L[1/S]\right)_{p^\infty}.
\end{equation}
Moreover, there is an exact sequence (see for example, \cite[p.~330 (2.6)]{CS05_CM})
\[
0\rightarrow \Y\left( \Zp(1)/\linf\right) \rightarrow \Z^2_S\left( \Zp(1)/\linf\right) \rightarrow \bigoplus_{v\in S(\linf)}\Zp \rightarrow \Zp\rightarrow 0
\]
and an isomorphism (see~\cite[\S I.6.1]{Rub00})
\begin{equation}\label{eqn:hydfinesel(1.2)}
R\left(\QQ_p/\Zp/{\linf}\right) \simeq \Hom\left(\Gal\left(M\left({\linf}\right)/{\linf}\right), \QQ_p/\Zp\right).
\end{equation}
Since the dual fine Selmer group is independent of the choice of $S$, it is not included in the notation.

\section{Fine Selmer Groups in the Cyclotomic Extension} \label{section:Finite Fine Selmer Groups in the Cyclotomic Extension}
The results in this section provide evidence for~\ref{conj:A}.
First, we prove that for a set of ordinary primes of positive density, the Selmer group is trivial over the cyclotomic $\Zp$-extension for rank $0$ elliptic curves.
Next, we provide evidence for \ref{conj:A} for a class of elliptic curves defined over $p$-rational number fields.

\subsection{Trivial Fine Selmer Groups in the Cyclotomic Tower}
\label{section: Trivial Fine Selmer Groups in the Cyclotomic Tower}
 
Throughout this section, assume that $\EC/\QQ$ is a rational elliptic curve.
Fix a number field $F$ and consider the base-change $\EC/F$ of the curve to $F$.
Given a prime number $p$, by slight abuse of notation, we denote by $F_{\cyc}/F$ the cyclotomic $\Zp$-extension and by $\Gamma=\Gal\left(F_{\cyc}/F\right)\simeq \Zp$ its Galois group, without mention of the prime $p$, as it can be inferred by the context.

At a prime $v$ in $F$, the reduction of $\EC$ modulo $v$ is denoted $\widetilde{\EC_v}$; it is a curve over the residue field $\kappa_v$.
Following \cite[Section~1(b)]{Maz72}, a prime $v\mid p$ is called \emph{anomalous} if $p$ divides $\abs{ \widetilde{\EC}_v(\kappa_v)}$.

In the remaining part of this section, we extend results of Greenberg~\cite[Proposition~5.1]{Gre99} and C.~Wuthrich~\cite[Section~9]{Wut07_fineSelmer} to base fields other than $\QQ$.
In Theorem~\ref{Theorem: Construction} we provide evidence for \ref{conj:A} for elliptic curves over a general number field.
We stress that the prime $p$ is not fixed in the remainder of this section and will vary over primes of good reduction.

In the statement of the next theorem we denote by $F^c$ the Galois closure of $F/\QQ$.
\begin{Th}
\label{Theorem: Construction}
Let $\EC/F$ be the base-change of a rational elliptic curve $\EC/\QQ$.
Suppose that it has rank $0$ over $F$ and that the Shafarevich--Tate group of $\EC/F$ is finite.
When $\EC$ has CM by an order in an imaginary quadratic field $K$, assume further that $F^c$ contains $K$.
Then the Selmer group $\Sel\left(\EC/F_{\cyc}\right)$ is trivial for a set of prime numbers of density at least $\frac{1}{[F^c:\QQ]}$.
In particular, \ref{conj:A} holds for $\EC/F$ at all such primes.
\end{Th}

\begin{proof} By assumption, the Selmer group over $F$ is finite since both the Mordell--Weil and the Shafarevich--Tate groups are finite.
If we further know that $p$ is a prime of good ordinary reduction for $\EC$, it follows from Mazur's Control Theorem that the cyclotomic $p$-primary Selmer group $\Sel\left(\EC/F_{\cyc}\right)$ is $\lg[\Gamma]$-cotorsion (see \cite[Corollary~4.9]{Gre_PCMS}).
In this setting, let $f_\EC(T)$ be a power series generating the characteristic ideal of $\X\left(\EC/F_{\cyc}\right)$.
Since $\Sel(\EC/F)$ is finite, $f_\EC(0)\neq 0$.
Denote by $c_v$ the local Tamagawa number at a prime $v$ and by $c_v^{(p)}$ the highest power of $p$ dividing it.
Then, \cite[Theorem~4.1]{Gre99} asserts that 
\begin{equation}\label{eqn: greenberg's char polynomial result}
f_\EC(0) \sim \left( \prod_{v \bad} c_v^{(p)}\right) \left( \prod_{v\mid p} \abs{ \widetilde{\EC}_v (\kappa_v)_p}^2\right) \abs{ \Sel(\EC/F)} \bigg/ \abs{ \EC(F)_p}^2
\end{equation}
where $a\sim b$ for $a, b\in \QQ_p^\times$ indicates that $a,b$ have the same $p$-adic valuation.

For a prime number $p$, consider the following five properties:
\begin{enumerate}[label=(\roman*)]
\item $p$ is a prime of good ordinary reduction for $\EC$;\label{pt:density_good_ord}
\item $\EC$ has no non-trivial $p$-torsion points defined over $F$;\label{pt:density_torsion}
\item $\EC/F$ has good ordinary reduction at all primes $v\mid p$ and all these primes are non-anomalous;\label{pt:density_anomalous}
\item the $p$-primary part $\Sha(\EC/F)_{p^\infty}$ of the Shafarevich--Tate group is trivial;\label{pt:density_sha}
\item $p$ does not divide the local Tamagawa number, i.e, $c_v^{(p)}=1$ for every prime $v$ of bad reduction.\label{pt:density_tamagawa}
\end{enumerate}
Since $\EC/F$ is assumed to have rank $0$, the condition $\EC(F)_p=0$ implies that $\Sel(\EC/F) = \Sha(\EC/F)_{p^\infty}$.
It follows from \eqref{eqn: greenberg's char polynomial result} that for a prime number satisfying \ref{pt:density_good_ord}--\ref{pt:density_tamagawa} above, $f_\EC(0)$ is a unit.

When $f_\EC(0)$ is a unit, elementary properties of characteristic power series show that $\X\left(\EC/F_{\cyc}\right)$ (and hence $\Y\left(\EC/F_{\cyc}\right)$) is finite, (see notation introduced in \eqref{eqn: dual selmer and fine selmer}).
Equivalently, both $\Sel\left(\EC/F_{\cyc}\right)$ and $R\left(\EC/F_{\cyc}\right)$ are finite.
When $\EC(F)_p=0$, \cite[Proposition~4.14]{Gre99} implies that $\X\left(\EC/F_{\cyc}\right)$ has no non-trivial finite $\lg[\Gamma]$-submodules.
In other words, $\X\left(\EC/F_{\cyc}\right)$ is trivial, whenever it is finite.
Thus, $\Y\left(\EC/F_{\cyc}\right)$ is also trivial.
Hence, \ref{conj:A} holds for $\EC/F$ when $\EC/F$ is an elliptic curve satisfying~\ref{pt:density_good_ord}--\ref{pt:density_tamagawa}.

To complete the proof, we show that for $\EC/F$ satisfying the assumptions of the theorem, properties~\ref{pt:density_good_ord}--\ref{pt:density_tamagawa} hold for a set of prime numbers of density at least $\frac{1}{[F^c:\QQ]}$.

When $\EC/\QQ$ is an elliptic curve without CM, we know by \cite[Th{\'e}or{\`e}me 20]{Ser81} that all primes in $\QQ$ outside a set of density~$0$ have good ordinary reduction.
When $\EC/F$ is an elliptic curve with CM by an order in $K$, Deuring's Criterion (see, for instance, \cite[Chapter~13,~\S4,~Theorem~12]{Lang_Elliptic_Functions}) asserts that the primes of ordinary reduction are those lying above rational primes that split in $K/\QQ$ and the density of such prime numbers equals $1/2$ by the Chebotarev density theorem.
Next, it follows from the celebrated result \cite[Th\'eor\`eme]{Mer96} of L.~Merel that for all but finitely many prime numbers, we have $\EC(F)_p =0$.
Assuming the finiteness of the Shafarevich--Tate group, condition \ref{pt:density_sha} holds for all but finitely many prime numbers, and the same is true for~\ref{pt:density_tamagawa} since the local Tamagawa number $c_v$ is equal to $1$ at the primes of good reduction.

The analysis of~\ref{pt:density_anomalous} requires more care.
By definition, a prime $v\mid p$ is anomalous when $a_v = 1 + \abs{\kappa_v} - \abs{ \widetilde{\EC}_v(\kappa_v)}$ is congruent to $1 \pmod{p}$.
Observe that by the Hasse bound, $\abs{a_v}\leq 2 \sqrt{\abs{\kappa_v}}$.
Therefore, if $v\mid p$ is a prime in $F$ that splits completely, so that $\kappa_v = \mathbb{F}_p$, then $a_v \equiv 1 \pmod{p}$ implies that $a_v=1$ for $p>5$.
By the Chebotarev density theorem, the density of rational primes that split completely in $F^c$ is $\frac{1}{[F^c:\QQ]}$.
Therefore, at least $\frac{1}{[F^c:\QQ]}$ of the primes in $\QQ$ split in $F$, as well.
By the previous discussion, the density of rational primes which split completely in $F$ and whose divisors are primes of good ordinary reduction for $\EC/F$ is at least $\frac{1}{[F^c:\QQ]}$.
Finally, since $\EC$ is defined over $\QQ$, the Modularity Theorem guarantees that $\EC$ is associated with an eigencuspform of weight $2$.
This allows us to appeal to the work of V.~K.~Murty \cite{Mur97}.
We conclude from \cite[pp.~288--289 or Theorem~5.1 and Remark~5.2]{Mur97} that for $\EC/\QQ$, the set of prime numbers with the property that $a_p=1$ has density $0$.
Since for all prime numbers $p$ that split completely and for all $v\mid p$, we have $a_v(\EC/F)=a_p(\EC/\QQ)$, we deduce that the set of prime numbers $p$ such that $a_v=1$ for at least one $v\mid p$ is a set of density $0$.
This completes the proof of the theorem.
\end{proof}

\begin{Remark}\leavevmode
\begin{enumerate}
\item 
It should be clear from the proof that one can insist that at all primes dividing the prime numbers in the set of positive density whose existence is stated in the theorem, the reduction type is good and ordinary.
\item The key difficulty in extending this result to elliptic curves defined over $F$ is that we rely on \cite{Mur97} to show that anomalous primes have density 0.
Since these results are proven for normalized weight $2$ eigencuspforms, we need to invoke the Modularity Theorem.
\end{enumerate}
\end{Remark}

An analogous statement can be proven in the supersingular case as well.
\begin{Th}
\label{Theorem: supersingular}
Let $\EC/\QQ$ be an elliptic curve, and suppose that $\Sel(\EC/F)$ is finite.
Then \ref{conj:A} holds for $\EC/F$ for all but finitely many primes of supersingular reduction.
\end{Th}

\begin{proof}
For an elliptic curve $\EC/F$ it is known that the Selmer group is not $\lg[\Gamma]$-cotorsion at a prime $p$ of supersingular reduction, see \cite[p.~19]{CS00_book}.
However, there is a notion of $\pm$-Selmer groups\footnote{We avoid giving the precise definition of these Selmer groups because their definition is intricate and also not relevant for the remainder of this paper.
For a precise definition, we refer the reader to \cite{Kob03} or \cite{BDKim13}.} when $p>3$, denoted by $\Sel^{\pm}\left(\EC/F_{\cyc}\right)$.
In the setting of the theorem, and under the additional hypothesis that $p>3$ is an unramified prime in $F$, it is known that $\Sel^{\pm}\left(\EC/F_{\cyc}\right)$ are $\lg[\Gamma]$-cotorsion, see \cite[first line of the proof of Corollary~3.15]{BDKim13}.
Therefore, in this case, we can define a pair of signed characteristic power series $f_\EC^{\pm}(T)$ for the Pontryagin duals $\X(\EC/F_{\cyc})^\pm$ of $\Sel^{\pm}\left(\EC/F_{\cyc}\right)$.
It follows from the definitions that the fine Selmer group is a subgroup of the signed Selmer groups.
To prove the theorem it thus suffices to show that either of the signed Selmer groups is finite for all but finitely many primes of good supersingular reduction as this will ensure that the fine Selmer group is also finite and its corresponding $\mu$ and $\lambda$ invariants vanish.

When $\Sel\left(\EC/F\right)$ is finite and $p>3$ is an unramified prime in $F$, we know from \cite[Theorem~1.2]{BDKim13} that 
\begin{equation}\label{eqn: ss char}
f_\EC^{\pm}(0) \sim \abs{\Sel(\EC/F)}\prod_{v \bad}c_v^{(p)}.
\end{equation}
If $f_\EC^{\pm}(0)\sim 1$, then it follows from the Structure Theorem that $\Sel^{\pm}\left(\EC/F_{\cyc}\right)$ are finite.
To complete the proof we show that $f_\EC^{\pm}(0)\sim 1$ for all but finitely many primes of good supersingular reduction.
\begin{enumerate}[label = (\roman*)]
\item Since $F$ is fixed, there are only finitely many primes which can ramify in $F$.
In other words, \eqref{eqn: ss char} holds for all but finitely many primes.
\item By assumption, $\Sel(\EC/F)$ is finite.
There are only finitely many primes which can divide its order.
\item The local Tamagawa number $c_v$ is equal to $1$ at the primes of good reduction.
Therefore, there are only finitely many primes which can divide ${\prod_{v \bad} c_v}$.
\end{enumerate}
Therefore, as $p$ varies over all supersingular primes of $\EC$, both signed Selmer groups $\Sel^{\pm}\left( \EC/F_{\cyc}\right)$ are finite for all but finitely many such primes.
Hence, $R\left(\EC/F_{\cyc}\right)$ is also finite for such $p$.
\end{proof}

\begin{Remark}
In fact, more is true.
\cite[Theorem~1.1 (or Theorem~3.14)]{BDKim13} applies in the setting of Theorem~\ref{Theorem: supersingular} and ensures that the $\X{^-}(\EC/F_{\cyc})$ does not contain any non-trivial finite index submodules.
Therefore, if $\Sel^{-}\left(\EC/F_{\cyc}\right)$ is finite, it must be trivial.
Since $R\left(\EC/F_{\cyc}\right)$ is a subgroup of $\Sel^{-}\left(\EC/F_{\cyc}\right)$, it must be trivial as well.
For the assertion that $\X(\EC/F_{\cyc})^+$ has no non-trivial finite index submodules, the additional hypothesis that $p$ is completely split in $F$ is required.
\end{Remark}

Combining Theorems \ref{Theorem: Construction} and \ref{Theorem: supersingular}, the next result is immediate.
\begin{cor}
\label{cor: conj A for CM}
Let $\EC$ be CM rational elliptic curve and let $\EC/F$ be its base-change to $F$.
Suppose that $\EC/F$ has rank $0$, that the Shafarevich--Tate group of $\EC/F$ is finite, and that the Galois closure $F^c$ of $F$ contains $K$.
Then \ref{conj:A} holds for $\EC/F$ for a set of prime numbers of density $\frac{1}{2}+\frac{1}{[F^c:\QQ]}$.
\end{cor}

\begin{proof}
By Deuring's Criterion we know that $1/2$ of the primes are supersingular and Theorem~\ref{Theorem: supersingular} asserts that there is a contribution of density $1/2$.
But, there is also a contribution from the primes of good ordinary reduction by Theorem~\ref{Theorem: Construction}.
The corollary follows.\end{proof}

Let us now turn to a special class of number fields, called $p$-rational number fields.

\subsection{\texorpdfstring{\ref{conj:A}}{} over \texorpdfstring{$p$}{}-Rational Number Fields}
\label{section: Conjecture A for p-rational}
For the number field $F$ and a fixed prime $p$, choose $S$ to be a finite set of primes of $F$ containing the primes above $p$ and the archimedean primes.
The weak Leopoldt conjecture for $\linf/F$ is the following assertion (see for example \cite[Theorem~10.3.22]{NSW08})
\begin{equation}\label{weak leopoldt classical}
H^2\left(\Gal\left(F_S/\linf\right), \QQ_p/\Zp\right)=0.
\end{equation}
It is known to hold for the cyclotomic $\Zp$-extension $F_{\cyc}/F$ (see \cite[Theorem~10.3.25]{NSW08}).
If \eqref{weak leopoldt classical} holds for a finite set $S$ as above, it also holds for the set $\Sigma= S_p \cup S_\infty$ (see \cite[Theorem~11.3.2]{NSW08}).
Therefore, the weak Leopoldt Conjecture is independent of the choice of $S$, when $S$ contains $\Sigma$.
Henceforth, fix $S=\Sigma$.
An equivalent formulation of the \ref{conj:Iwasawa} for $F$ is the assertion that $\mathcal{G}_\Sigma\left(F_{\cyc}\right)=\Gal\left(F_\Sigma(p)/F_{\cyc}\right)$ is a free pro-$p$ group (see \cite[Theorem~11.3.7]{NSW08}).
Moreover, a pro-$p$ group $G$ is free if and only if its $p$-cohomological dimension $\cd_p(G)$ is less or equal to $1$ (see \cite[Corollary~3.5.17]{NSW08}).
Combining these results with \cite[Chapter~I, Section~4, Proposition~21]{Ser_GalCoho}, one obtains the following equivalent formulation: 
\begin{equation}\label{eqn: NSW 11.3.7}
\text{the~\ref{conj:Iwasawa} for } F \text{ is true } \Longleftrightarrow H^2\left(\mathcal{G}_\Sigma\left(F_{\cyc}\right), \ZZ/p\ZZ\right)=0.
\end{equation}

To state the results in this section, we recall the notion of a special class of number fields, called \emph{$p$-rational}, which were introduced in~\cite{MN90}.
We refer the reader to \cite[Theorem~IV.3.5 and Definition~IV.3.4.4]{Gra13} for a detailed discussion.
\begin{Defi}
Denote by $F_{S_p}$ the maximal extension of $F$ unramified outside $S_p$ and let $F_{S_p}(p)/F$ be its maximal pro-$p$ sub-extension.
Set $\mathcal{G}_{S_p}(F) = \Gal\left(F_{S_p}(p)/F\right)$.
If $\mathcal{G}_{S_p}(F)$ is free pro-$p$, then $F$ is called $p$-rational.
\end{Defi}

Some examples of $p$-rational fields include:
\begin{enumerate}[label = (\roman*)]
\item the field $\QQ$ of rational numbers;
\item imaginary quadratic fields such that $p$ does not divide the class number (see \cite[Proposition~4.1.1]{Gre16});
\item cyclotomic fields $\QQ(\mu_{p^n})$, where $p$ is a regular prime and $n\geq 1$ (combine \cite[Example~II.7.8.1.1]{Gra13} with \cite[Proposition~13.22]{Was97});
\item more generally, number fields $F$ containing $\mu_p$ with the property that $\# S_p(F)=1$ and such that $p$ does not divide the class number of $F$ (see \cite[Theorem~3.5-(iii)]{Gra13}).
\end{enumerate}
$p$-rational number fields have been studied by Greenberg in \cite{Gre16}, where he explains heuristic reasons to believe that a number field $F$ should be $p$-rational for all primes outside a set of density~$0$ (see \cite[{\S}7.4.4]{Gre16}).
In \cite[Table~4.1]{BR20}, R.~Barbulescu and J.~Ray provide examples of non-abelian $p$-rational number fields.

The following result is easily deduced from the aforementioned results in Galois cohomology.
A proof is included for the sake of completeness.
\begin{Th}
Let $F$ be a $p$-rational number field.
Then the following assertions hold.
\begin{enumerate}[label = \textup{(\roman*\textup)}]
\item \label{thm: classical Conj for p rational}
The~\ref{conj:Iwasawa} holds for $F$.\label{pt:thm_rational:Iwa}
\item Suppose that $F$ contains $\mu_p$ and that $\EC/F$ is an elliptic curve such that $\EC(F)_p\neq 0$.
Then \ref{conj:A} holds for $\EC/F$.\label{pt:thm_rational:conjA}
\end{enumerate}
\end{Th}
\begin{proof}\leavevmode
\begin{itemize}
\item[\ref{pt:thm_rational:Iwa}] Since $p\neq 2$, we can replace $S_p$ by $\Sigma$ in the definition of $p$-rational fields.
This is because the archimedean primes are unramified in $F_{S_p}(p)/F$ when $p$ is odd.
By definition, if $F$ is $p$-rational, $\mathcal{G}_\Sigma(F)=\Gal\left(F_{\Sigma}(p)/F\right)$ has $p$-cohomological dimension at most 1.
Hence
\[
H^2\left(\mathcal{G}_\Sigma(F), \ZZ/p\ZZ\right)=0.
\] 
Since $\mathcal{G}_\Sigma(F_{\cyc})=\Gal(F_{\Sigma}(p)/F_{\cyc})$ is a closed normal subgroup of $\mathcal{G}_\Sigma(F)$, it follows from \cite[Proposition~3.3.5]{NSW08} that
\[
\cd_p\left(\mathcal{G}_\Sigma(F_{\cyc})\right)\leq \cd_p\left(\mathcal{G}_\Sigma(F)\right) \leq 1.
\]
Thus $H^2(\mathcal{G}_\Sigma(F_{\cyc}), \ZZ/p\ZZ)=0$, and the result follows from \eqref{eqn: NSW 11.3.7}.
\item[\ref{pt:thm_rational:conjA}] Since $F\supseteq \mu_p$ and $\EC(F)_p\neq 0$ by assumption, the Weil pairing ensures that $F(\EC_p)/F$ is either trivial or of degree $p$.
Thus, $F(\EC_{p^\infty})/F$ is pro-$p$.
The theorem follows from the first point together with Theorem~\ref{thm: cs05 thm 3.4}.\qedhere
\end{itemize}
\end{proof}

\section{\texorpdfstring{\ref{conj:B}}{} for Elliptic Curves with CM: Special Cases}
\label{section: Special Cases}

In this section, we provide evidence for \ref{conj:B}.
First, in Section~\ref{S: from PR} we provide sufficient conditions for \ref{conj:B} to hold when $p$ is a prime of good ordinary reduction, see Theorem~\ref{theorem: main theorem CM}.
In Section~\ref{reformulation} we give a different formulation of \ref{conj:B} for CM elliptic curves and prove cases of the conjecture when $p$ is a prime of good supersingular reduction.
We start with a lemma about good reduction of CM elliptic curves that can be found extracted from \cite[proof of Theorem~5.7-(i)]{Rub_AWS}.

\begin{Lemma}
\label{pot good red}
Let $F$ be a number field and let $\EC/F$ be an elliptic curve with CM by an order inside the ring of integers $\OK$ of an imaginary quadratic field $K$.
Let $p$ be an odd prime number and suppose that the following hypotheses hold:
\begin{enumerate}[label = \textup{(}\roman*\textup{)}]
\item $\EC$ has good reduction at all primes above $p$.
\item The Galois group $G=\Gal(F_\infty/F)$ is isomorphic to $\Zp^2$, where $F_\infty$ denotes $F(\EC_{p^\infty})$.
\end{enumerate}
Then $\EC$ has good reduction everywhere over $F$.
\end{Lemma}

\begin{proof}
It follows from the theory of complex multiplication that $F$ contains the Hilbert class field $K'$ of $K$.
Since the extension $F_\infty/F$ is a $p$-extension and $[F(\EC_p):F]$ is prime-to-$p$, it follows that $F=F(\EC_p)$.
Therefore, $K'(\EC_p)\subseteq F$.

Since all primes above $p$ are of good reduction, we only need to check that at primes away from $p$, the curve $\EC$ has good reduction.
This follows from the criterion of N{\'e}ron--Ogg--Shafarevich, because every such prime is unramified in the $\ZZ_p^2$-extension $F_\infty/F$.
\end{proof}

\subsection{} 
\label{S: from PR}
Fix a number field $F$.
We will work in the following setting.

\begin{equation}\tag*{\textup{\textbf{Ass~1}}}\label{assmpn: for PR case}
\begin{minipage}{0.9\textwidth}
\begin{enumerate}[label=\textup{(}\roman*\textup{)}]
  \item $p\neq 2,3$ is a fixed prime which splits in an imaginary quadratic field $K$;
  \item $\EC$ is an elliptic curve defined over $F$ with CM by $\OK$, and $K$ is contained in $F$;
  \item $\EC$ has good reduction at primes above $p$;
  \item the Galois group $G=\Gal(F_\infty/F)$ is isomorphic to $\Zp^2$, where $F_\infty$ denotes $F(\EC_{p^\infty})$.
\end{enumerate}
\end{minipage}
\end{equation}
In the setting of~\ref{assmpn: for PR case}, write $H=\Gal(F_\infty/F_{\cyc})$, and fix a finite set $S$ containing $S_p \cup S_{\infty}$.

Note that~\ref{assmpn: for PR case} ensures that $\EC$ has good ordinary reduction at $p$, see \cite[Chapter~13 Theorem~12 (Deuring's Criterion)]{Lang_Elliptic_Functions}.
Observe that given any $p$-adic Lie group~$\mathcal{G}$ and a finitely generated $\lg[\mathcal{G}]$-module $M$, the group $M_{\mathcal{G}} := H_0\left(\mathcal{G},M\right)$ is finitely generated as a $\ZZ_p$-module.

\begin{Lemma}
\label{lem: pseudo-isomorphism}
Suppose that \ref{assmpn: for PR case} holds.
Then, the following map of $\lg[H]$-modules is a pseudo-isomorphism, \emph{i.~e.} it has a finite kernel and cokernel,
\[
\label{eqn: psedo-isomorphism}
\Y\left(\EC/F_\infty\right)_H \rightarrow \Y\left(\EC/F_{\cyc}\right).
\]
\end{Lemma}
\begin{proof}
Let $L$ be a finite extension of $F$ contained in $F_S$.
For each $v\in S$, write $W_v(L) = {\bigoplus_{w\mid v}} \EC(L_w)\otimes\QQ_p/ \Zp$.
We have the maps
\begin{align*}
 r_{\cyc}\colon &\Sel(\EC/F_{\cyc}) \longrightarrow \bigoplus_{v\mid p} W_v(F_{\cyc})
 \intertext{ and }
 r_{\infty}\colon &\Sel(\EC/F_\infty) \longrightarrow \bigoplus_{v\mid p} W_v(F_{\infty})
\end{align*}
where $W_v(F_{\cyc})$ (\emph{resp.}~$W_v(F_{\infty}$)) is the direct limit of $W_v(L)$ with respect to the restriction map as $L$ ranges over all finite extensions of $F$ contained in $F_{\cyc}$ (\emph{resp.}~$F_{\infty}$).
Write $C(F_{\cyc})$ (\emph{resp.}~$C(F_{\infty})$) for the image of $r_{\cyc}$ (\emph{resp.}~$r_{\infty}$).
Consider the following diagram
\[
\begin{tikzcd}
 0 \arrow[r] & R(\EC/F_{\cyc})_p \arrow[r] \arrow[d, "\alpha"] & \Sel(\EC/F_{\cyc})_p \arrow[r] \arrow[d, "\beta"] & C(F_{\cyc})_p \arrow[d, "\gamma"] \arrow[r] & 0\\ 
 0 \arrow[r] & R(\EC/F_{\infty})_p^H \arrow[r] & \Sel(\EC/F_{\infty})_p^H \arrow[r] & C(F_{\infty})_p^H
 \end{tikzcd}
\]
Note that $\beta$ is an isomorphism (see \cite[Lemme~1.1(\romannumeral 1) and Lemme~1.3]{PR81}).
Therefore $\ker(\beta)$ and $\coker(\beta)$ are trivial; hence $\ker(\alpha)=0$.
Further, observe that there is an inclusion
\[
\ker \gamma \subseteq \ker\left(\bigoplus_{v\mid p} K_v\left(F_{\cyc}\right) \xrightarrow{\delta_v} K_v\left(F_\infty\right)^H\right).
\]
Now, observe that 
\[
\bigoplus_{v\mid p} \ker(\delta_v) = \bigoplus_{v\mid p} H^1\left(H_v , \EC(F_{\infty,v})_{p^\infty}\right).
\]
This latter object is known to be finite by using an argument identical to \cite[proof of Lemma~4.2]{CS05}.
Therefore, by the snake lemma, $\coker(\alpha)$ must be finite.
\end{proof}

\begin{Remark}
Another way to prove this lemma was pointed out to us by the referee.
Consider the fundamental diagram
\[
\begin{tikzcd}
 0 \arrow[r] & R(\EC/F_{\cyc})_p \arrow[r] \arrow[d, "\alpha"] & H^1\left(G_S(F_{\cyc}), \EC_{p^\infty}\right) \arrow[r] \arrow[d, "\beta"] & \bigoplus_{v\in S} K_v\left( F_{\cyc}\right) \arrow[d, "\gamma"] \arrow[r] & 0\\ 
 0 \arrow[r] & R(\EC/F_{\infty})_p^H \arrow[r] & H^1\left(G_S(F_{\infty}), \EC_{p^\infty}\right)^H \arrow[r] & \bigoplus_{v\in S} K_v\left( F_{\infty}\right)^H.
 \end{tikzcd}
\]
The map $\beta$ is the restriction map; it is surjective and $\ker(\beta) = H^1\left( H, \EC(F_{\infty})_{p^\infty}\right)$.
Similarly, $\ker(\gamma) = \bigoplus_{v\in S_p} H^1\left( H_v, \EC(F_{\infty,v})_{p^\infty}\right)$.
To show that $\ker(\gamma)$ is finite, we use an argument similar to \cite[proof of Lemma~4.2]{CS05}.
First, recall a result of H.~Imai \cite[Theorem]{Im75} which asserts that $\EC(F_{\cyc,v})_{p^\infty}$ is finite.
Note that $H_v\simeq \Zp$ and $\EC(F_{\infty,v})_{p^\infty}^\vee$ is a torsion $\lg[H_v]$-module (since it is in fact finitely generated over $\Zp$).
It follows that $H^1\left( H_v, \EC(F_{\infty,v})_{p^\infty}\right)$ is also finite.
In fact, $H^1\left( H_v, \EC(F_{\infty,v})_{p^\infty}\right)=0$ which can be proven in the same way as \cite[Lemma~5.4]{CSW01} using the fact that $H_v$ has $p$-cohomological dimension $1$.
Furthermore, since $\EC(F_{\cyc})_{p^\infty}$ is finite by a result of K.~Ribet \cite[Theorem~1]{Rib81}, the global version of the above argument ensures that $\ker(\beta)$ is also finite, see also \cite[pp.~834--835]{CS05}.
Applying the snake lemma, the lemma follows.
\end{Remark}

Since $\EC$ is an elliptic curve with CM, both $G$ and $H$ are abelian.
Under the assumption that $G\simeq \Zp^2$, we further know that $\lg[H]\simeq \Zp\llbracket T\rrbracket$.
We now state an equivalent condition for a $\lg$-module to be pseudonull.

\begin{prop}
\label{ven03 P5.4}
Let $M$ be a finitely generated $\lg$-module which is also finitely generated as a $\lg[H]$-module.
Then the module $M$ is $\lg$-torsion.
Further, $M$ is $\lg[H]$-torsion if and only if it is $\lg$-pseudonull.
\end{prop}

\begin{proof}
Note that $G\simeq H\times \Gamma$ where $\Gamma\simeq \Zp$.
The first assertion follows from the fact that $\lg$ is not finitely generated over $\lg[H]$.
The second assertion is a special case of \cite[Proposition~5.4]{Ven03}.
\end{proof}

\begin{Lemma}
\label{lemma: finite implies pseudonull} Let $M$ be a finitely generated $\lg[G]$-module which is also finitely generated over $\lg[H]$.
If $M_H$ is finite, then $M$ is a pseudonull $\lg$-module.
\end{Lemma}

\begin{proof}
We are grateful to the referee for suggesting the following proof, which is simpler than the one we had in a first version of our manuscript.
Since $H\cong \ZZ_p$, it follows from the structure theory of $\lg[H]$-modules that whenever $M_H$ is finite, then $M$ is torsion over $\lg[H]$.
The conclusion of the lemma is now immediate from Proposition~\ref{ven03 P5.4}.
\end{proof}

The main theorem of this section is the following.
\begin{Th}
\label{theorem: main theorem CM}
Suppose that \ref{assmpn: for PR case} holds.
If $\Y(\EC/F_{\cyc})$ is finite, then \ref{conj:B} holds for $(\EC,F_{\infty})$.
\end{Th}
\begin{proof}
By Lemma~\ref{lem: pseudo-isomorphism}, if $\Y(\EC/F_{\cyc})$ is finite, then so is $\Y\left(\EC/F_{\infty}\right)_H$.
The theorem follows from Lemma~\ref{lemma: finite implies pseudonull}.
\end{proof}

\begin{Remark}
We point out that for a given prime $p$, we cannot conclude that \ref{conj:B} holds for $(\EC,F_{\infty})$ for a rank $0$ elliptic curve $\EC/F$ with CM by combining Theorems~\ref{Theorem: Construction} and \ref{theorem: main theorem CM}.
This is because, in the proof of Theorem~\ref{Theorem: Construction} it was required that the elliptic curve does not admit any non-trivial $p$-torsion point over $F$.
However, in proving Theorem~\ref{theorem: main theorem CM}, we assume that $F_\infty/F$ is a pro-$p$ extension; hence $F$ must contain non-trivial $p$-torsion points.
\end{Remark}

Another case where we can show \ref{conj:B} is the following.

\begin{prop}
\label{prop: Howson}
Suppose that \ref{assmpn: for PR case} holds.
Further assume that $\mathfrak{X}(\EC/F_{\cyc})$ is a finitely generated $\Zp$-module of $\ZZ_p$-rank $2$ and that $\EC(F_\infty)$ has a point of infinite order.
Then~\ref{conj:B} holds for $(\EC,F_{\infty})$.
\end{prop}
\begin{proof}
By \ref{assmpn: for PR case}, we know that $\EC/F$ has good reduction everywhere.
Next, it follows from \cite[Theorem~2.8]{How02} that
\[\label{eq:howson}
\rank_{\lg[H]}\mathfrak{X}(\EC/F_\infty) = \rank_{\Zp} \mathfrak{X}(\EC/F_{\cyc}).
\]
We explain this briefly.
To apply \cite[Theorem~2.8]{How02} one must assume that Conjecture~2.5 \emph{ibid.} holds.
As mentioned on p.~649 \emph{ibid.}, this conjecture is equivalent to Conjecture~2.6~\emph{ibid.} when all primes above $p$ have good ordinary reduction.
This conjecture predicts that $\mathfrak{X}(\EC/F_{\cyc})$ is $\lg[\Gamma]$-torsion and our hypothesis that $\mathfrak{X}(\EC/F_{\cyc})$ is a finitely generated $\Zp$-module accounts for it.
For the final assumption in Theorem~2.8 \emph{ibid,} the inclusion $\mu_p\subseteq F$ is ensured by the Weil pairing.

It is shown in \cite[Theorem~4.5-(ii)]{CS05} that
\[
\rank_{\lg[H]}\mathfrak{Y}(\EC/F_\infty)\leq \rank_{\lg[H]}\mathfrak{X}(\EC/F_\infty) -2.
\]
and our assumption, combined with~\eqref{eq:howson}, implies that $\rank_{\lg[H]}\mathfrak{Y}(\EC/F_\infty)=0$, showing that $\mathfrak{Y}(\EC/F_\infty)$ is $\lg[G_{F_\infty/F}]$-pseudonull.
\end{proof}

\begin{Remark}\label{rmk:howson}
We are grateful to the referee for the following observation.
In\cite[Theorem~4.5-(i)]{CS05} it is shown that if $\rank_{\lg[H]}\mathfrak{X}(\EC/F_\infty)$ is odd, then
\[
\rank_{\lg[H]}\mathfrak{Y}(\EC/F_\infty)\leq \rank_{\lg[H]}\mathfrak{X}(\EC/F_\infty) -1.
\]
Since $\mathfrak{X}(\EC/F_{\cyc})$ is a finitely generated $\Zp$-module, if $\rank_{\Zp}\mathfrak{X}(\EC/F_{\cyc})= 1$, it would follow that $\mathfrak{Y}(\EC/F_\infty)$ is $\lg[G_{F_\infty/F}]$-pseudonull.
Unfortunately, though, one can show that under our assumptions $\rank_{\lg[H]}\mathfrak{X}(\EC/F_\infty)$ is always even, and therefore the above argument cannot be used to show~\ref{conj:B} in more cases.
The argument forcing $\rank_{\lg[H]}\mathfrak{X}(\EC/F_\infty)$ to be even comes from CM theory: indeed, by functoriality, the $\lg[H]$-module $\mathfrak{X}(\EC/F_\infty)$ comes endowed with a structure of an $\OK$-module, and therefore it is ultimately a $\lg[H]\otimes \OK$-module.
For brevity, denote $\lg[H]\otimes\OK$ by $\lgK[H]$, and set $\qh=\Frac(\lg[H]),\qhK=\Frac(\widetilde{\lg[H]})$.
Clearly, $\lgK[H]$ is a finite extension of $\lg[H]$ of rank $2$: in particular, it is integral so that $\qhK=\lgK[H]\otimes_{\lg[H]}\qh$ and $\dim_{\qh}\qhK=2$.
Now, by definition,
\begin{align*}
2\cdot\rank_{\lgK[H]}\mathfrak{X}(\EC/F_\infty)&=2\cdot\dim_{\qhK}\Bigl(\mathfrak{X}(\EC/F_\infty)\otimes_{\lgK[H]}\qhK\Bigr)\\
&=2\cdot\dim_{\qhK}\Bigl(\mathfrak{X}(\EC/F_\infty)\otimes_{\lgK[H]}\lgK[H]\otimes_{\lg[H]}\qh)\\
&=2\cdot\dim_{\qhK}\Bigl(\mathfrak{X}(\EC/F_\infty)\otimes_{\lg[H]}\qh)\\
&=\dim_{\qh}\Bigl(\mathfrak{X}(\EC/F_\infty)\otimes_{\lg[H]}\qh)\\
&=\rank_{\lg[H]}\mathfrak{X}(\EC/F_\infty),
\end{align*}
showing that $\rank_{\lg[H]}\mathfrak{X}(\EC/F_\infty)$ is indeed even.
\end{Remark}

\subsection{Reformulation of \texorpdfstring{\ref{conj:B}}{}}
\label{reformulation}
Let $\EC/F$ be an elliptic curve, and let $\linf$ be an $S$-admissible $p$-adic Lie extension containing the trivializing extension $F_\infty$.
Throughout this section we suppose that \ref{conj:A} holds for $\EC/F$.
Since $G_S(\linf)$ acts trivially on $\EC_{p^\infty}$, \ref{conj:B} for $(\EC,\linf)$ has an equivalent formulation in terms of the pseudonullity of the Galois group $\Gal(M(\linf)/\linf)$, where $M(\linf)$ is the maximal unramified abelian pro-$p$ extension of $\linf$ such that all primes above $p$ in $\linf$ split completely.

\begin{Reform*}[{see~\cite[p.~827]{CS05}}]
Let $\EC/F$ be an elliptic curve, and let $\linf$ be an $S$-admissible, $p$-adic Lie extension over $F$ such that $G_S\left(\linf\right)$ acts trivially on $\EC_{p^\infty}$.
Then $\Y\left(\Zp(1)/{\linf}\right)$ is $\lg[G_{\linf/F}]$-pseudonull.
\end{Reform*}

The next result asserts that for an $S$-admissible $p$-adic Lie extension $\linf/F$ containing $F_\infty$, the $\lg[G_{\linf/F}]$-pseudonullity of the Iwasawa module $X_{\nr}^\linf$ is equivalent to the pseudonullity of a certain quotient module.
(The notation $X_{\nr}^\linf$ was introduced at the beginning of this section).
This result is well-known to experts and follows from results available in the literature, see for example \cite[Theorem~4.9]{Ven03_Compositio}.
For the convenience of the reader, a proof is provided here purely relying on techniques that are more germane to our paper.

\begin{Th}
\label{thm:appendix}
Let $\EC/F$ be an elliptic curve with CM by an order in an imaginary quadratic field $K$ such that $K\subseteq F$ and suppose that $\Gal(F_\infty/F)\simeq \Zp^2$.
Let $\linf/F$ be an abelian $S$-admissible $p$-adic Lie extension containing $F_\infty$.
Then, the following statements are equivalent
\begin{enumerate}[label = \textup{(\alph*)}]
 \item The~\ref{conj:Iwasawa} is true for $F$ and $X_{\nr}^{\linf}$ is $\lg[G_{\linf/F}]$-pseudonull.\label{pt:them_appendix:mu}
 \item \ref{conj:B} holds for $(\EC,\linf)$.\label{pt:them_appendix:conjB}
 \item The~\ref{conj:Iwasawa} is true for $F$ and $\Y(\Zp(1)/\linf)$ is $\lg[G_{\linf/F}]$-pseudonull.\label{pt:them_appendix:Iwa}
\end{enumerate}
\end{Th}

\begin{proof}Since $\EC/F$ has CM by the imaginary quadratic field $K$ contained in $F$ and $F_\infty/F$ is a $\ZZ_p^2$-extension, it follows that $F$ contains $K'(\EC_p)$ where $K'$ is the Hilbert class field of $K$ (see Lemma~\ref{pot good red}).
Moreover, since $\linf/F$ is an abelian extension containing $F_\infty$ and, by definition of being admissible, it contains no element of order $p$, it must be a $\Zp^d$-extension for some $d\geq 2$.
It follows that the only primes that can ramify in this extension are the primes above $p$ and therefore we can assume that $S=S_p \cup S_\infty$.

\medskip
\emph{Equivalence of~\ref{pt:them_appendix:mu} and of~\ref{pt:them_appendix:Iwa}}: 
We need to show that
\[
\Y\left(\Zp(1)/\linf\right) \text{ is } \lg[G_{\linf/F}]\text{-pseudonull } \Longleftrightarrow X_{\nr}^{\linf} \text{ is } \lg[G_{\linf/F}]\text{-pseudonull}.
\]
Write $X_{\cs}^{\linf}$ to denote the Galois group $\Gal\left(M(\linf)/\linf\right)$.
It is known by the work of U.~Jannsen (see for example~\cite[Proposition~4.7-(ii)]{Ven03_Compositio}) that there is an exact sequence
\[
\bigoplus_{v\in S_{\cs} \cup S_{\ram}} \Ind_{G_{\linf/F}}^{G_{\linf/F, v}}\left( \ZZ_p\right) \longrightarrow X_{\nr}^{\linf} \longrightarrow X_{\cs}^{\linf} \longrightarrow 0.
\]
Here, $S_{\cs}$ denotes the set of non-archimedean primes in $S$ which are completely split in $\linf/F$ and $S_{\ram}$ denotes the set of non-archimedean primes in $S$ which are ramified in $\linf/F$.
Note that in our setting $S_{\cs}=\emptyset$ because every prime above $p$ is finitely decomposed in $F_{\cyc}/F$, and $S_{\ram}=S_p$.
Since the base field $F$ contains $\mu_p$ it follows from~\eqref{eqn:hydfinesel(1.2)} that
\[
X_{\cs}^{\linf} = \Gal\left(M(\linf)/\linf\right) \simeq \Y\left(\Zp(1)/\linf\right).
\]
Therefore, to complete the proof of the equivalence it is enough to show that $X_{\nr}^{\linf}$ and $X_{\cs}^{\linf}$ are pseudo-isomorphic.
In other words, it suffices to prove that
\begin{equation*}
\bigoplus_{v\in S_p} \ZZ_p\llbracket G_{\linf/F}\rrbracket \otimes_{\ZZ_p\llbracket G_{\linf/F,v}\rrbracket} \ZZ_p = \bigoplus_{v\in S_p} \Ind_{G_{\linf/F}}^{G_{\linf/F,v}}\left( \ZZ_p\right)
\end{equation*}
is a $\lg[G_{\linf/F}]$-pseudonull module.
We know from \cite[Th\'eor\`eme~3.2]{LN00} (observe that since $F$ contains $K'(\EC_p)$, condition $(\mathrm{i})$ \emph{ibid.}~is satisfied, by the Weil pairing) that for all $v\mid p$, the decomposition group at $v$ inside $G_{\linf/F}$ has dimension at least $2$.
It follows that $\bigoplus_{v\in S_p}\Ind_{G_{\linf/F}}^{G_{\linf/F, v}}\left( \ZZ_p\right)$ is $\lg[G_{\linf/F,v}]$-pseudonull.
This completes the proof of the equivalence.

\medskip
\emph{Equivalence of~\ref{pt:them_appendix:conjB} and~\ref{pt:them_appendix:Iwa}}: 
It follows from the discussion in \cite[p.~825]{CS05} that
\begin{equation}\label{eq:tensorE}
\Y\left( \EC/\linf\right) \simeq \left(\Y( \Zp(1)/\linf ) \otimes \EC_{p^\infty}^\vee.\right)
\end{equation}
Here $G_{\linf/F}$ acts diagonally on the tensor product and $\EC_{p^\infty}^\vee$ is a $\Zp$-module with a $G_{\linf/F}$-action induced by the $G_S(F)$-action.
This latter action makes sense because $F_\infty$ is the trivializing extension of $\EC_{p^\infty}$.
In this setting, \ref{conj:A} for $\EC/F$ is equivalent to the \ref{conj:Iwasawa} for $F$ (see Theorem~\ref{thm: cs05 thm 3.4}).
Therefore, using~\cite[Proposition~2.12]{Ven03_Compositio} and~\cite[Proposition~3.4]{OV02}, the isomorphism in~\eqref{eq:tensorE} yields that $\Y\left( \EC/\linf\right)$ is $\lg[G_{\linf/F}]$-pseudonull if and only if $\Y\left( \Zp(1)/\linf\right)$ is $\lg[G_{\linf/F}]$-pseudonull.
\end{proof}

\begin{Remark}
In the above proof, the assumption that $\EC$ has CM is not really used.
We detail in \S\ref{subsec:non_commutative} the proof in the general case.
\end{Remark}

We now prove a special case of \ref{conj:B} in the supersingular reduction setting and provide applications pertaining to universal norms.
For the remainder of this section, we work in the following setting:
\begin{equation}\tag*{\textup{\textbf{Ass~2}}}\label{assmpn: for supersingular case}
\begin{minipage}{0.9\textwidth}
\begin{enumerate}[label=\textup{(}\roman*\textup{)}]
  \item $K$ is an imaginary quadratic field of class number 1;
  \item $\EC$ is an elliptic curve defined over $K$, and with CM by $\OK$;
  \item $p$ is an odd prime of good supersingular reduction for $\EC$;
  \item $p$ does not divide the order of the $S_p$-class group of $F=K(\EC_p)$;
\end{enumerate}
\end{minipage}
\end{equation}

\begin{Remark}
It follows from~\ref{assmpn: for supersingular case}-(ii) that the Galois group $G=\Gal(F_\infty/F)$ is isomorphic to $\Zp^2$ and we write 
$H=\Gal(F_\infty/F_{\cyc})$.
\end{Remark}
Recall that the $i$-th Iwasawa cohomology group over $F_\infty$ is defined, when $S=S_p$, as
\[
\Z^i_{S_p}\left( \Zp(1)/F_\infty\right) = \varprojlim_L H^i_{\text{\'et}}\left( \mathcal{O}_L[1/p], \Zp(1)\right),
\]
where $L$ ranges over all finite extensions of $F$ contained in $F_\infty$.

\begin{prop}
\label{prop: combined prop as per referee's suggestions}
Suppose that \ref{assmpn: for supersingular case} holds.
Then $\Z^2_{S_p}\left( \Zp(1)/F_{\infty}\right)=0$.
In particular, \ref{conj:B} holds for $(\EC,F_\infty)$ and $\mu_{p^\infty}(F)$ is a universal norm from $F_\infty$.
\end{prop}

\begin{proof}
We are grateful to the referee for suggesting the following proof, which is simpler than the one we had in a first version of our manuscript.
Using the Poitou--Tate sequence over $F$ as in~\cite[p.~553 \S 2.4-(1)]{Kat06}, we have that
\begin{equation}\label{eq:PoitouTate_S_p}
0 \longrightarrow \Cl_{S_p}(F)_{p^\infty} \longrightarrow \Z^2_{S_p}\left( \Zp(1)/F\right) \longrightarrow \bigoplus_{v\in S_p(F)} \Zp \longrightarrow \Zp \longrightarrow 0.
\end{equation}
Since $p$ is a prime of supersingular reduction for $\EC/K$, we know that there exists a unique prime above $p$ in $K$.
Moreover, $p$ is totally ramified in the extension $\Gal(F_\infty/K)$, see for example \cite[Section~1]{PR04}.
In particular, there is a unique prime above $p$ in $F$, \emph{i.~e.} $\abs{S_p}=1$.
Combining this with the assumption that the $p$-Sylow subgroup of the $S_p$-class group is trivial yields, through~\eqref{eq:PoitouTate_S_p}, that $\Z^2_{S_p}\left( \Zp(1)/F\right)=0$.
By Nekovar's spectral sequence (see~\cite[Corollary~8.4.8.4-(ii)]{Nek06}), we obtain that
\[
\Z^2_{S_p}\left( \Zp(1)/F_{\infty}\right)_G \simeq \Z^2_{S_p}\left( \Zp(1)/F\right) =0.
\]
Now, employing Nakayama's Lemma we conclude that $\Z^2_{S_p}\left( \Zp(1)/F_{\infty}\right)=0$.

Next, consider the exact sequence (see, for example,~\cite[p.~330 (2.6)]{CS05_CM})
\[
0 \longrightarrow \Y\left( \Zp(1)/F_{\infty}\right) \longrightarrow \Z^2_{S_p}\left( \Zp(1)/F_{\infty}\right) \longrightarrow \bigoplus_{v\in S_p(F_{\infty})} \Zp \longrightarrow \Zp \longrightarrow 0.
\]
It follows from the first part of the proof that $\Y\left( \Zp(1)/F_{\infty}\right)=0$.
Moreover, the same descent argument as above using Nekovar's spectral sequence shows that $\Y\left( \Zp(1)/F_{\cyc}\right)=X_{\cs}^{F_{\cyc}}=0$: in particular, $\mu(X_{\cs}^{F_{\cyc}})=0$.
By \cite[Corollary~11.3.16]{NSW08}, we know that $\mu(X_{\cs}^{F_{\cyc}})=\mu(X_{\nr}^{F_{\cyc}})=0$.
Therefore, we obtain that~\ref{conj:Iwasawa} holds for $F$, and we can apply Theorem~\ref{thm:appendix}, showing that~\ref{conj:B} holds.

To prove the final assertion, consider the exact sequence (see \cite[p.~335 (3.26)]{CS05_CM})
\begin{align*}
0\longrightarrow H_2\left(G, \Z^2_{S_p}\left(\Zp(1)/F_\infty\right) \right) \longrightarrow \Z^1_{S_p}\left( \Zp(1)/F_\infty\right)_G &\xrightarrow{\tau_{F_\infty/F}} \Z^1_{S_p}\left(\Zp(1)/F\right)\\ &\longrightarrow H_1\left(G, \Z^2_{S_p}\left(\Zp(1)/F_\infty\right) \right) \longrightarrow 0.
\end{align*}
We have shown above that $\Z^2_{S_p}\left(\ZZ_p(1)/F_\infty\right)=0$; hence, $\tau_{F_\infty/F}$ is an isomorphism.
It follows that $\mu_{p^\infty}(F)$ is a universal norm from $F_\infty$ (see \cite[Corollary~3.27]{CS05_CM} for details).
\end{proof}

The following corollary provides asymptotics for the growth of the $p$-primary torsion of the fine Selmer group at each layer of the $\Zp^2$-extension.

\begin{cor}
Suppose that \ref{assmpn: for PR case} holds and that either
\begin{enumerate}[label=\textup{(\roman*)}]
\item $\Y(\EC/F_{\cyc})$ is finite; or\label{cor_asympt:finite}
\item $\mathfrak{X}(\EC/F_{\cyc})$ is a finitely generated $\Zp$-module of $\ZZ_p$-rank equal to $2$ and $\EC(F_\infty)$ has a point of infinite order.\label{cor_asympt:2}
\end{enumerate}
Then
\[
\ord_p\Bigl( R\bigl(\EC/F(\EC_{p^n})\bigr)^{\vee}[p^\infty]\Bigr) = O\left(p^{n}\right).
\]
If, moreover, \ref{assmpn: for supersingular case} holds, then 
$\ord_p\left( R\left(\EC/F(\EC_{p^n})\right)^{\vee}[p^\infty]\right)=0$.
\end{cor}

\begin{proof}
\ref{conj:A} holds by assumption in each case and \ref{conj:B} holds by Theorem~\ref{theorem: main theorem CM} in case~\ref{cor_asympt:finite} and by Proposition~\ref{prop: Howson} in case~\ref{cor_asympt:2}
The first claim follows from~\cite[Corollary~6.14]{KL21}.

When \ref{assmpn: for supersingular case} holds, a better estimate can be obtained, and we thank the referee for this observation.
Since $F_\infty$ is the trivializing extension, we have
\[
\Z^2_{S_p}(T_p\EC/F_\infty) \simeq \Z^2_{S_p}(\Zp(1)/F_\infty) \otimes T_p\EC
\]
and thus Proposition~\ref{prop: combined prop as per referee's suggestions} implies that $\Z^2_{S_p}(T_p\EC/F_\infty)=0$.
By Nekovar's spectral sequence, we also know that
\[
\Z^2_{S_p}\left(T_p\EC/F(\EC_{p^n})\right) \simeq \Z^2_{S_p}(T_p\EC/F_\infty)_{G_n} = 0.
\]
Since $\Y\left(\EC/F(\EC_{p^n})\right)$ is contained in $\Z^2_{S_p}\left(T_p\EC/F(\EC_{p^n})\right)$, the result follows.
\end{proof}
\subsection{The noncommutative setting}\label{subsec:non_commutative}
Even though this paper largely treats the commutative case, an analogue of Theorem~\ref{thm:appendix} is valid even in the noncommutative setting.
We would like to thank the referee for pointing this out, and insisting that the general case be included.
Let $S$ be a finite set of primes of $F$ containing the primes above $p$ , the archimedean primes and the primes of bad reduction for $\EC$.
\begin{Th}
\label{non-commutative}
Let $\EC/F$ be an elliptic curve without complex multiplication and $p\geq 5$ be a rational prime.
Assume that $F$ contains the $p$-torsion points of $\EC$.
Suppose that the elliptic curve has either potential ordinary or potential multiplicative reduction at all the primes $v \mid p$, that $F$ contains the $p$-torsion points of $\EC$, and let $F_\infty=F(\EC_{p^\infty})/F$ be the trivializing extension.
Write $S = S_p \cup S_{\bad} \cup S_\infty$ and set $G=\Gal(F_{\infty}/F), H=\Gal(F_{\infty}/F_{\cyc})$.
Then the following assertions are equivalent:
\begin{enumerate}[label = \textup{(\alph*)}]
 \item The~\ref{conj:Iwasawa} is true for $F$ and $X_{\nr}^{F_{\infty}}$ is $\lg$-pseudonull.\label{pt:noncomm_appendix:mu}
 \item \ref{conj:B} holds for $(\EC,F_{\infty})$.\label{pt:noncomm_appendix:conjB}
 \item The~\ref{conj:Iwasawa} is true for $F$ and $\Y(\Zp(1)/F_\infty)$ is $\lg$-pseudonull.\label{pt:noncomm_appendix:Iwa}
\end{enumerate}
\end{Th}
\begin{proof} As in the proof of Theorem~\ref{thm:appendix}, the equivalence between~\ref{pt:noncomm_appendix:mu} and of~\ref{pt:noncomm_appendix:Iwa} will follow once we prove that $X_{\nr}^{F_\infty}$ and $X_{\cs}^{F_\infty}$ are pseudo-isomorphic. We know from \cite[Lemma~2.8]{Coa99} that $G_{F_\infty/F,v}$ has dimension at least $2$ at primes $v\in S$; here we have used the fact that when $v$ does not divide $p$, it is not possible in our setting that $\EC$ has bad but potential good reduction (see for example, \cite[the paragraph after Theorem~5.2]{OV02}).
It follows that $\bigoplus_{v\in S}\Ind_{G_{F_\infty/F}}^{G_{F_\infty/F, v}}\left( \ZZ_p\right)$ is $\lg[G_{F_\infty/F,v}]$-pseudonull.

We next prove the implication~\ref{pt:noncomm_appendix:Iwa}~$\Rightarrow$~\ref{pt:noncomm_appendix:conjB}: assume that assertion \ref{pt:noncomm_appendix:Iwa} holds.
Then, by \cite[Lemma~3.8, p.~825]{CS05}, there is an isomorphism 
\[
\Y(E/{F_{\infty}})\simeq \Y\left(\Zp(1)/F_\infty\right)\otimes \EC_{p^\infty}^\vee,
\]
where the tensor product is over $\Zp,$ and the action of $\Gal(F(\EC_{p^\infty})/F)$ on the right hand-side is the diagonal action.
The hypothesis that the~\ref{conj:Iwasawa} holds for $F$ assures us of the finite generation of $\Y\left(\Zp(1)/F_\infty\right)$ 
as a $\lg[H]$-module.
Hence, to show that assertion~\ref{pt:noncomm_appendix:conjB} holds, it suffices to prove that $\Y(\EC/F_{\infty})$ is a torsion $\lg[H]$-module.

The pseudonullity of $\Y(\EC/F_{\infty})$ as a $\lg$-module is equivalent to $\Y(\EC/F_{\infty})$ being torsion as a $\lg[H]$-module.
By our hypothesis that $\Y\left(\Zp(1)/F_\infty\right)$ is pseudonull, it follows that it is also finitely generated and torsion as a $\lg[H]$-module.
Thus there exists a finite, free $\lg[H]$-resolution of $\Y\left(\Zp(1)/F_\infty\right)$ such that the alternating sum of the $\lg[H]$-ranks of the free modules in the resolution is zero.
Tensoring such a resolution over $\Zp$ with $\EC_{p^\infty}^\vee$ preserves
the exactness and gives a 
finite, free $\lg[H]$-resolution of $\Y\left(\Zp(1)/F_\infty\right)\otimes \EC_{p^\infty}^\vee.$ Further, the alternating sum of the
$\lg[H]$-ranks of the free modules is still zero, whence $\Y\left(\Zp(1)/F_\infty\right)\otimes \EC_{p^\infty}^\vee$ is $\lg[H]$-torsion.
This proves assertion~\ref{pt:noncomm_appendix:conjB}.
Note that this argument also proves the equality of the $\lg[H]$-ranks of $\Y(\ZZ_p(1)/F_{\infty})$ and $\Y(\EC/F_\infty)$.

It remains to prove the implication~\ref{pt:noncomm_appendix:conjB}~$\Rightarrow$~\ref{pt:noncomm_appendix:Iwa}.
Suppose that~\ref{conj:B} is true for $(\EC,F_{\infty}).$ Then~\ref{conj:A} is true
and the dual fine Selmer group $\Y(\EC/F_{\cyc})$ of $\EC$ over the cyclotomic $\ZZ_p$-extension is a finitely generated $\Zp$-module.
The vanishing of the Iwasawa $\mu$-invariant for $F_{\cyc}$ is a consequence of~\cite[Theorem~3.4]{CS05}.
Now suppose that  $\Y(\ZZ_p(1)/F_{\infty})$ is not $\lg[H]$-torsion, and  hence has positive rank as a $\lg[H]$-module.
By the remark above on the equality of $\lg[H]$-ranks, this implies that $\Y(\EC/F_{\infty})$  also has positive $\lg[H]$-rank, contradicting the hypothesis. This completes the proof of the equivalence.
\end{proof}

\section{\texorpdfstring{\ref{conj:B}}{} and the \texorpdfstring{\ref{conj:GGC}}{}}
\label{section actually relating ggc to conj B}
The aim of this section is to clarify the connection between the \ref{conj:GGC} and \ref{conj:B} for CM elliptic curves.
For the sake of brevity, we henceforth refer to the \ref{conj:GGC} as \ref{GGC-short}.

Both conjectures pertain to the pseudonullity of certain Iwasawa modules.
Even though \ref{conj:B} was proposed as a generalization of \ref{GGC-short}, the precise formulation of this connection is rather intricate.
Using Theorem~\ref{thm:appendix}, we make precise in which sense \ref{conj:B} for CM elliptic curves is a generalization of \ref{GGC-short} (see Theorem~\ref{Thm: conj B implies GGC in several}).

Fix an imaginary quadratic field $K$ and denote its Hilbert class field by $K'$.
Given an elliptic curve $\EC/K'$ with CM by an order in $K$, set
\begin{align*}
&F= K'(\EC_p), \quad F_\infty = K'(\EC_{p^\infty})=F(\EC_{p^\infty}), \\
&G=\Gal(F_\infty/F), \quad
\mathcal{G}_\infty =\Gal(F_\infty/K), \quad \mathcal{G}'_\infty =\Gal(F_\infty/K').
\end{align*}
Note that $G\simeq \Zp^2$.
Set $\widetilde{K}$ (\emph{resp.}~$\widetilde{K'}$, $\widetilde{F}$) to be the compositum of all $\Zp$-extensions of $K$ (\emph{resp.}~of $K'$, of $F$).
Since the Leopoldt conjecture is true for imaginary quadratic fields, $\widetilde{K}$ is the unique $\Zp^2$-extension of $K$.
For the rest of this section, we make the following assumption.
\begin{equation}\tag*{\textup{\textbf{Ass~3}}}\label{assmpn: unramified}
\begin{minipage}{0.9\textwidth}
\begin{enumerate}[label=\textup{(}\roman*\textup{)}]
\item $p$ is an odd prime that is unramified in $K$;
\item the prime $p$ is such that $K'\cap \widetilde{K} =K$.
\end{enumerate}
\end{minipage}
\end{equation}
By the theory of complex multiplication, $\mathcal{G}_\infty = G\times \Delta$ and $\mathcal{G}'_\infty = G\times \Delta'$ where $\Delta \simeq \Gal(F/K)$ (\emph{resp.}~$\Delta' \simeq \Gal(F/K')$) is a finite abelian group.
Recall from \cite[Remark on p.~IV-13]{Ser97} that $\Delta'$ is a Cartan subgroup of $\GL_2(\mathbb{F}_p)$ and hence it either has order $p^2-1$ or $(p-1)^2$: in any case, $p\nmid \abs{\Delta'}$.

\begin{Remark}\label{rmk:Ass3}\leavevmode
\begin{enumerate}
\item In fact, it is forced by \ref{assmpn: unramified}-(i) that $\mu_p \not\subset K$.
This can be seen as follows: the only pair $(p,K)$ for which $\mu_p\subset K$ is when $p=3$ and $K=\QQ(\sqrt{-3})$; but this contradicts~\ref{assmpn: unramified}-(i).\label{pt:rmk:Ass3:p_3}
\item We now discuss~\ref{assmpn: unramified}-(ii) in a little more detail.
This assumption is trivially satisfied when $p$ does not divide the class number of $K$.
But observe that, in general, $K'\cap \widetilde{K}$ is contained in the anti-cyclotomic $\Zp$-extension of $K$, denoted by $K_{\ac}$.
For a proof of this fact, see \cite[Lemma~2.2]{Fuj13}.
Therefore, \ref{assmpn: unramified}-(ii) is equivalent to the following condition:
\[
\text{(ii$'$) The prime } p \text{ is such that } K'\cap K_{\ac} =K.
\]
To know more about non-trivial examples where this condition is satisfied, we refer the reader to \cite{Bri07}.
For a specific example, see Example~4 \emph{ibid}.
Moreover, \ref{assmpn: unramified}-(ii) is closely related to the notion of $p$-rationality (see \cite[p.~2133]{Bri07}) but we will not discuss this point any further.
\end{enumerate}
\end{Remark}

Set the notation $K'_\infty$ to denote the composite of the fields $K'$ and $\widetilde{K}$.
The theory of complex multiplication guarantees that $F_\infty =F\widetilde{K} = FK'_{\infty}$.
Recall that $F_\infty$ is the trivializing extension for the Galois representation associated to $T_p\EC$ and it is an $S$-admissible $p$-adic Lie extension.
We note that $F_\infty \subseteq \widetilde{F}$.

Denote by $L(\widetilde{F})$ (\emph{resp.}~${L(F_\infty)}$) the maximal abelian unramified pro-$p$-extension of $\widetilde{F}$ (\emph{resp.}~of $F_\infty$).
Denote by $\scrF_S$ the maximal abelian pro-$p$ extension of $\widetilde{F}$ unramified outside $S$.
Set the notation
\begin{equation}\label{notation for galois groups}
X_{\nr}^{\widetilde{F}}=\Gal\left(L(\widetilde{F})/\widetilde{F}\right), \quad X_{\nr}^{F_\infty}=\Gal\left({L(F_\infty)}/{F_\infty}\right), \quad X_S^{\widetilde{F}} = \Gal\left(\scrF_S/\widetilde{F}\right).
\end{equation}
As in the previous sections, given any extension $\linf/F$, we denote by $M(\linf)$ the maximal unramified abelian $p$-extension of $\linf$ where all primes above $p$ in $\linf$ split completely; this group is related to the fine Selmer group (see \eqref{eqn:hydfinesel(1.2)}).
For most of the discussion, $\linf$ will either be $F_\infty$ or $\widetilde{F}$.
For convenience, the diagram of fields is drawn in Figure~\ref{fig:fields_diagram}.
\begin{figure}
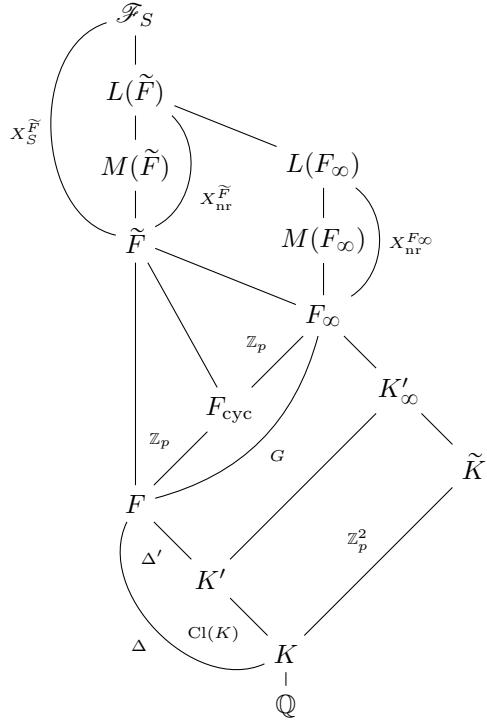

  \centering
  \usebox{\fieldsdiagram}
  \caption{The diagram of fields occurring in Theorem~\ref{Thm: conj B implies GGC in several}}
  \label{fig:fields_diagram}
\end{figure}

Recall the statement of \ref{GGC-short} for $F$ (the statement for $K$ is analogous, by replacing $F,\widetilde{F}, \lg[G_{\widetilde{F}/F}]$ by $K,\widetilde{K}, \lg[G_{\widetilde{K}/K}$], respectively).

\begin{labelledconj}{\textup{GGC}}{GGC-short}
With notation as above, $X_{\nr}^{\widetilde{F}}$ is a pseudonull $\lg[G_{\widetilde{F}/F}]$-module.
\end{labelledconj}

The following results are required to relate \ref{GGC-short} to the pseudonullity of the fine Selmer group.
The first lemma assures pseudonullity over a larger tower, once it holds for a proper subextension.

\begin{Lemma}[Pseudonullity Lifting Lemma]
\label{lem:lifting}
Let $n\geq 3$, let $\mathcal{F}/\QQ$ be a finite Galois extension containing $\mu_p$, and denote by $\widetilde{\mathcal{F}}$ the compositum of all $\Zp$-extensions of $\mathcal{F}$.
Suppose that $\Gal(\widetilde{\mathcal{F}}/\mathcal{F})\simeq \Zp^n$ and let $\mathcal{F}^{(d)}\subsetneq \widetilde{\mathcal{F}}$ be such that $\Gal(\mathcal{F}^{(d)}/\mathcal{F})\simeq \Zp^d$ for some $2\leq d < n$.
If $X_{\nr}^{\mathcal{F}^{(d)}}$ is $\lg[G_{\mathcal{F}^{(d)}/\mathcal{F}}]$-pseudonull then \ref{GGC-short} holds for $\widetilde{\mathcal{F}}/\mathcal{F}$.
\end{Lemma}
\begin{proof}
This lemma is a special case of \cite[Theorem~12]{Ban07}.
Since $\mathcal{F}$ contains $\mu_p$, the technical conditions in the mentioned theorem are satisfied by \cite[Theorem 3.2]{LN00} or \cite[Remark~15]{Ban07}.
\end{proof}

The next result studies pseudonullity of Galois modules under base change.

\begin{Lemma}[Pseudonullity Shifting Down Lemma]
\label{lem:shifting}
Let $\mathcal{F}$ be a number field and let $\mathcal{F}^{(d)}/\mathcal{F}$ be a $\Zp^d$-extension.
Suppose that $\mathcal{F}_1/\mathcal{F}$ is a finite extension and set $\mathcal{K}= \mathcal{F}_1\cdot \mathcal{F}^{(d)}$.
If $X_{\nr}^{\mathcal{K}}$ is a $\lg[G_{\mathcal{K}/\mathcal{F}_1}]$-pseudonull module, then $X_{\nr}^{\mathcal{F}^{(d)}}$ is a $\lg[G_{\mathcal{F}^{(d)}/\mathcal{F}}]$-pseudonull module.
\end{Lemma}

\begin{proof}
For a proof, see \cite[Theorem~3.1-(i)]{Kle16}.
\end{proof}

The purpose of the next result is to show that \ref{conj:B} is indeed a generalization of \ref{GGC-short}.
We resume the notation introduced at the beginning of this section.

\begin{Th}
\label{Thm: conj B implies GGC in several}
In the setting of~\ref{assmpn: unramified}, suppose that there exists an elliptic curve $\EC/K'$ with CM by an order in $K$ such that \ref{conj:B} holds for $(\EC,F_\infty)$.
Then \ref{GGC-short} holds for $K$.
\end{Th}

\begin{proof}
Let $\EC/K'$ be an elliptic curve with CM by an order in $K$ such that~\ref{conj:B} holds for $(\EC, F_\infty)$.
Regarding it as being defined over $F=K'(\EC_p)$, Theorem~\ref{thm:appendix} shows that $X_{\nr}^{F_\infty}$ is $\lg[G_{F_\infty/F}]$-pseudonull.

Applying Lemma~\ref{lem:shifting} with $\mathcal{F}=K'$, $\mathcal{F}_1 = F$, $\mathcal{F}^{(2)}=K'_\infty$, and $\mathcal{K}=F_\infty= FK_{\infty}^\prime$, the $\lg[G_{F_\infty/F}]$-pseudonullity of $X_{\nr}^{F_\infty}$ can be shifted down to $\lg[G_{K'_\infty/K'}]$-pseudonullity of $X_{\nr}^{K^\prime_\infty}$.
Therefore, we have shown that
\[
\text{\ref{conj:B} for } (\EC, F_\infty) \Longrightarrow X_{\nr}^{K^\prime_\infty} \text{ is } \lg[G_{K'_\infty/K'}]\text{-pseudonull}.
\]
Another application of Lemma~\ref{lem:shifting} with $\mathcal{F}=K$, $\mathcal{F}_1 = K'$, $\mathcal{F}^{(2)} = \widetilde{K}$, and $\mathcal{K}= K'_\infty = K'\widetilde{K}$, shows that $\lg[G_{K'_\infty/K'}]$-pseudonullity of $X_{\nr}^{K'_\infty}$ can be shifted down to $\lg[G_{\widetilde{K}/K}]$-pseudonullity of $X_{\nr}^{\widetilde{K}}$.
This is \ref{GGC-short} for $K$.
\end{proof}

\begin{cor}
\label{cor for GGC of K'Ep}
With the same hypotheses of Theorem~\ref{Thm: conj B implies GGC in several},~\ref{GGC-short} holds also for any number field $L$ such that $K'(\mu_p) \subseteq L \subseteq F$.
\end{cor}

\begin{proof}
Applying Lemma~\ref{lem:shifting} with $\mathcal{F}=L$, $\mathcal{F}_1 = F$, $\mathcal{F}^{(2)}=LK_{\infty}^\prime$, and $\mathcal{K}=F_\infty = F\mathcal{F}^{(2)}$, the $\lg[G_{F_\infty/F}]$-pseudonullity of $X_{\nr}^{F_\infty}$ obtained in Theorem~\ref{thm:appendix} can be shifted down to $\lg[G_{\mathcal{F}^{(2)}/\mathcal{F}}]$-pseudonullity of $X_{\nr}^{\mathcal{F}^{(2)}}$.
Therefore, we have shown that
\[
\text{\ref{conj:B} for } (\EC, F_\infty) \Longrightarrow X_{\nr}^{\mathcal{F}^{(2)}} \text{ is } \lg[G_{\mathcal{F}^{(2)}/L}]\text{-pseudonull}.
\]
As discussed in Remark~\ref{rmk:Ass3}-\eqref{pt:rmk:Ass3:p_3}, $\mu_p\not\subset K$, hence $L\neq K$ and $L$ admits at least two complex embeddings.
Letting $\widetilde{L}$ denote the compositum of all $\ZZ_p$-extensions of $L$,~\cite[Theorem~13.4]{Was97} implies that $\Gal(\widetilde{L}/L)\cong \ZZ_p^n$ for some $n\geq 3$.
Using Lemma~\ref{lem:lifting} with $\mathcal{F}=L$, pseudonullity of $X_{\nr}^{\mathcal{F}^{(2)}}$ as a $\lg[G_{\mathcal{F}^{(2)}/L}]$-module implies \ref{GGC-short} for $L$.
\end{proof}

\section*{Acknowledgements}
D.~K.~is supported by a PIMS Postdoctoral fellowship.
S.~R.~is supported by the NSERC Discovery Grant 2019-03987.
We thank the referee for their timely and thorough reading of the manuscript which helped strengthen some of the results and also contributed to a clearer exposition.
\bibliographystyle{acm}
\bibliography{references}

\begin{thebibliography}{10}

\bibitem{Ari14}
{\sc Aribam, C.~S.}
\newblock On the {$\mu$}-invariant of fine {S}elmer groups.
\newblock {\em J. Number Theory 135\/} (2014), 284--300.

\bibitem{Ban07}
{\sc Bandini, A.}
\newblock Greenberg's conjecture and capitulation in
  {$\mathbb{Z}_p^d$}-extensions.
\newblock {\em J. Number Theory 122}, 1 (2007), 121--134.

\bibitem{BR20}
{\sc Barbulescu, R., and Ray, J.}
\newblock Numerical verification of the {C}ohen-{L}enstra-{M}artinet heuristics
  and of {G}reenberg's {$p$}-rationality conjecture.
\newblock {\em J. Th\'{e}or. Nombres Bordeaux 32}, 1 (2020), 159--177.

\bibitem{Bha07}
{\sc Bhave, A.}
\newblock Analogue of {K}ida's formula for certain strongly admissible
  extensions.
\newblock {\em J. Number Theory 122}, 1 (2007), 100--120.

\bibitem{BCG+15}
{\sc Bleher, F.~M., Chinburg, T., Greenberg, R., Kakde, M., Pappas, G.,
  Sharifi, R., and Taylor, M.~J.}
\newblock Higher {C}hern classes in {I}wasawa theory.
\newblock {\em Amer. J. Math. 142}, 2 (2020), 627--682.

\bibitem{Bri07}
{\sc Brink, D.}
\newblock Prime decomposition in the anti-cyclotomic extension.
\newblock {\em Math. Comp. 76}, 260 (2007), 2127--2138.

\bibitem{Coa99}
{\sc Coates, J.}
\newblock Fragments of the {${\rm GL}_2$} {I}wasawa theory of elliptic curves
  without complex multiplication.
\newblock In {\em Arithmetic theory of elliptic curves ({C}etraro, 1997)},
  vol.~1716 of {\em Lecture Notes in Math.} Springer, Berlin, 1999, pp.~1--50.

\bibitem{CoaSchSuj03}
{\sc Coates, J., Schneider, P., and Sujatha, R.}
\newblock Modules over {I}wasawa algebras.
\newblock {\em J. Inst. Math. Jussieu 2}, 1 (2003), 73--108.

\bibitem{CS00_book}
{\sc Coates, J., and Sujatha, R.}
\newblock {\em Galois cohomology of elliptic curves}, vol.~88 of {\em Tata
  Institute of Fundamental Research Lectures on Mathematics}.
\newblock Published by Narosa Publishing House, New Delhi; for the Tata
  Institute of Fundamental Research, Mumbai, 2000.

\bibitem{CS05_CM}
{\sc Coates, J., and Sujatha, R.}
\newblock Fine {S}elmer groups for elliptic curves with complex multiplication.
\newblock In {\em Algebra and number theory}. Hindustan Book Agency, Delhi,
  2005, pp.~327--337.

\bibitem{CS05}
{\sc Coates, J., and Sujatha, R.}
\newblock Fine {S}elmer groups of elliptic curves over {$p$}-adic {L}ie
  extensions.
\newblock {\em Math. Ann. 331}, 4 (2005), 809--839.

\bibitem{CSW01}
{\sc Coates, J., Sujatha, R., and Wintenberger, J.-P.}
\newblock On the {E}uler-{P}oincar\'{e} characteristics of finite dimensional
  {$p$}-adic {G}alois representations.
\newblock {\em Publ. Math. Inst. Hautes \'{E}tudes Sci.}, 93 (2001), 107--143.

\bibitem{DdSMS03}
{\sc Dixon, J.~D., du~Sautoy, M. P.~F., Mann, A., and Segal, D.}
\newblock {\em Analytic pro-{$p$} groups}, second~ed., vol.~61 of {\em
  Cambridge Studies in Advanced Mathematics}.
\newblock Cambridge University Press, Cambridge, 1999.

\bibitem{FW79}
{\sc Ferrero, B., and Washington, L.~C.}
\newblock The {I}wasawa invariant {$\mu _{p}$} vanishes for abelian number
  fields.
\newblock {\em Ann. of Math. (2) 109}, 2 (1979), 377--395.

\bibitem{Fuj13}
{\sc Fujii, S.}
\newblock On a bound of {$\lambda$} and the vanishing of {$\mu$} of
  {$\mathbb{Z}_p$}-extensions of an imaginary quadratic field.
\newblock {\em J. Math. Soc. Japan 65}, 1 (2013), 277--298.

\bibitem{Fuj17}
{\sc Fujii, S.}
\newblock On {G}reenberg's generalized conjecture for {CM}-fields.
\newblock {\em J. Reine Angew. Math. 731\/} (2017), 259--278.

\bibitem{Gra13}
{\sc Gras, G.}
\newblock {\em Class field theory: From theory to practice}.
\newblock Springer Monographs in Mathematics. Springer-Verlag, Berlin, 2003.
\newblock Translated from the French manuscript by Henri Cohen.

\bibitem{Gre99}
{\sc Greenberg, R.}
\newblock Iwasawa theory for elliptic curves.
\newblock In {\em Arithmetic theory of elliptic curves ({C}etraro, 1997)},
  vol.~1716 of {\em Lecture Notes in Math.} Springer, Berlin, 1999,
  pp.~51--144.

\bibitem{Gre_PCMS}
{\sc Greenberg, R.}
\newblock Introduction to {Iwasawa} theory for elliptic curves.
\newblock {\em Arithmetic algebraic geometry 9\/} (2001), 407--464.

\bibitem{Gre01}
{\sc Greenberg, R.}
\newblock Iwasawa theory---past and present.
\newblock In {\em Class field theory---its centenary and prospect ({T}okyo,
  1998)}, vol.~30 of {\em Adv. Stud. Pure Math.} Math. Soc. Japan, Tokyo, 2001,
  pp.~335--385.

\bibitem{Gre16}
{\sc Greenberg, R.}
\newblock Galois representations with open image.
\newblock {\em Ann. Math. Qu\'{e}. 40}, 1 (2016), 83--119.

\bibitem{How_thesis}
{\sc Howson, S.}
\newblock {\em Iwasawa theory of Elliptic Curves for $p$-adic {L}ie
  extensions}.
\newblock PhD thesis, University of Cambridge, 1998.

\bibitem{How02}
{\sc Howson, S.}
\newblock Euler characteristics as invariants of {I}wasawa modules.
\newblock {\em Proc. London Math. Soc. (3) 85}, 3 (2002), 634--658.

\bibitem{Im75}
{\sc Imai, H.}
\newblock A remark on the rational points of abelian varieties with values in
  cyclotomic {$Z_{p}$}-extensions.
\newblock {\em Proc. Japan Acad. 51\/} (1975), 12--16.

\bibitem{Iwa73}
{\sc Iwasawa, K.}
\newblock On {${\bf Z}_{l}$}-extensions of algebraic number fields.
\newblock {\em Ann. of Math. (2) 98\/} (1973), 246--326.

\bibitem{Iwa73_AC}
{\sc Iwasawa, K.}
\newblock On the {$\mu $}-invariants of {$Z_{\ell}$}-extensions.
\newblock In {\em Number theory, algebraic geometry and commutative algebra, in
  honor of {Y}asuo {A}kizuki}. Kinokuniya, Tokyo, 1973, pp.~1--11.

\bibitem{Jha12}
{\sc Jha, S.}
\newblock Fine {S}elmer group of {H}ida deformations over non-commutative
  {$p$}-adic {L}ie extensions.
\newblock {\em Asian J. Math. 16}, 2 (2012), 353--365.

\bibitem{SJ11}
{\sc Jha, S., and Sujatha, R.}
\newblock On the {H}ida deformations of fine {S}elmer groups.
\newblock {\em J. Algebra 338\/} (2011), 180--196.

\bibitem{Kat06}
{\sc Kato, K.}
\newblock Universal norms of {$p$}-units in some non-commutative {G}alois
  extensions.
\newblock {\em Doc. Math.}, Extra Vol. (2006), 551--565.

\bibitem{BDKim13}
{\sc Kim, B.-D.}
\newblock The plus/minus {S}elmer groups for supersingular primes.
\newblock {\em J. Aust. Math. Soc. 95}, 2 (2013), 189--200.

\bibitem{Kle16}
{\sc Kleine, S.}
\newblock Relative extensions of number fields and {G}reenberg's generalised
  conjecture.
\newblock {\em Acta Arith. 174}, 4 (2016), 367--392.

\bibitem{Kob03}
{\sc Kobayashi, S.}
\newblock Iwasawa theory for elliptic curves at supersingular primes.
\newblock {\em Invent. Math. 152}, 1 (2003), 1--36.

\bibitem{KL21}
{\sc Kundu, D., and Lim, M.~F.}
\newblock Control theorems for fine {S}elmer groups.
\newblock {\em J. th{\'e}or. Nombres Bordeaux\/} (2022), accepted for
  publication.

\bibitem{Lang_Elliptic_Functions}
{\sc Lang, S.}
\newblock {\em Elliptic functions}, second~ed., vol.~112 of {\em Graduate Texts
  in Mathematics}.
\newblock Springer-Verlag, New York, 1987.
\newblock With an appendix by J. Tate.

\bibitem{LN00}
{\sc Lannuzel, A., and Nguyen Quang~Do, T.}
\newblock Conjectures de {G}reenberg et extensions pro-{$p$}-libres d'un corps
  de nombres.
\newblock {\em Manuscripta Math. 102}, 2 (2000), 187--209.

\bibitem{LP19_pseudonull}
{\sc Lei, A., and Palvannan, B.}
\newblock Codimension two cycles in {I}wasawa theory and elliptic curves with
  supersingular reduction.
\newblock {\em Forum Math. Sigma 7\/} (2019), Paper No. e25, 81.

\bibitem{Maz72}
{\sc Mazur, B.}
\newblock Rational points of abelian varieties with values in towers of number
  fields.
\newblock {\em Invent. Math. 18\/} (1972), 183--266.

\bibitem{McC01}
{\sc McCallum, W.~G.}
\newblock Greenberg's conjecture and units in multiple
  {$\mathbb{Z}_p$}-extensions.
\newblock {\em Amer. J. Math. 123}, 5 (2001), 909--930.

\bibitem{MS03}
{\sc McCallum, W.~G., and Sharifi, R.~T.}
\newblock A cup product in the {G}alois cohomology of number fields.
\newblock {\em Duke Math. J. 120}, 2 (2003), 269--310.

\bibitem{Mer96}
{\sc Merel, L.}
\newblock Bornes pour la torsion des courbes elliptiques sur les corps de
  nombres.
\newblock {\em Invent. Math. 124}, 1-3 (1996), 437--449.

\bibitem{Min86}
{\sc Minardi, J.~V.}
\newblock {\em Iwasawa modules for {$\mathbb{Z}^d_p$}-extensions of algebraic
  number fields}.
\newblock ProQuest LLC, Ann Arbor, MI, 1986.
\newblock Thesis (Ph.D.)--University of Washington.

\bibitem{MN90}
{\sc Movahhedi, A., and {N}{guyen Quang Do}, T.}
\newblock Sur l'arithm\'{e}tique des corps de nombres {$p$}-rationnels.
\newblock In {\em S\'{e}minaire de {T}h\'{e}orie des {N}ombres, {P}aris
  1987--88}, vol.~81 of {\em Progr. Math.} Birkh\"{a}user Boston, Boston, MA,
  1990, pp.~155--200.

\bibitem{Mur97}
{\sc Murty, V.~K.}
\newblock Modular forms and the {C}hebotarev density theorem. {II}.
\newblock In {\em Analytic number theory ({K}yoto, 1996)}, vol.~247 of {\em
  London Math. Soc. Lecture Note Ser.} Cambridge Univ. Press, Cambridge, 1997,
  pp.~287--308.

\bibitem{Nek06}
{\sc Nekov\'{a}\v{r}, J.}
\newblock Selmer complexes.
\newblock {\em Ast\'{e}risque}, 310 (2006), viii+559.

\bibitem{NSW08}
{\sc Neukirch, J., Schmidt, A., and Wingberg, K.}
\newblock {\em Cohomology of number fields}, second~ed., vol.~323 of {\em
  Grundlehren der mathematischen Wissenschaften [Fundamental Principles of
  Mathematical Sciences]}.
\newblock Springer-Verlag, Berlin, 2008.

\bibitem{NQD}
{\sc Nguyen Quang~Do, T.}
\newblock Analogues sup\'{e}rieurs du noyau sauvage.
\newblock {\em S\'{e}m. Th\'{e}or. Nombres Bordeaux (2) 4}, 2 (1992), 263--271.

\bibitem{DV05}
{\sc Nguyen Quang~Do, T., and Vauclair, D.}
\newblock {$K_2$} et conjecture de {G}reenberg dans les
  {$\mathbb{Z}_p$}-extensions multiples.
\newblock {\em J. Th\'{e}or. Nombres Bordeaux 17}, 2 (2005), 669--688.

\bibitem{NucSuj21}
{\sc N{uccio~Mortarino~Majno~di~Capriglio}, F.~A.~E., and Sujatha, R.}
\newblock Residual supersingular {I}wasawa theory and signed {I}wasawa
  invariants.
\newblock {\em Rendiconti del Seminario Mathematico di Padova\/} (2021),
  accepted for publication, DOI~10.4171/RSMUP/111.

\bibitem{Och09}
{\sc Ochi, Y.}
\newblock A remark on the pseudo-nullity conjecture for fine {S}elmer groups of
  elliptic curves.
\newblock {\em Comment. Math. Univ. St. Pauli 58}, 1 (2009), 1--7.

\bibitem{OV02}
{\sc Ochi, Y., and Venjakob, O.}
\newblock On the structure of {S}elmer groups over {$p$}-adic {L}ie extensions.
\newblock {\em J. Algebraic Geom. 11}, 3 (2002), 547--580.

\bibitem{Oza01}
{\sc Ozaki, M.}
\newblock Iwasawa invariants of {$\mathbb{Z}_p$}-extensions over an imaginary
  quadratic field.
\newblock In {\em Class field theory---its centenary and prospect ({T}okyo,
  1998)}, vol.~30 of {\em Adv. Stud. Pure Math.} Math. Soc. Japan, Tokyo, 2001,
  pp.~387--399.

\bibitem{PR81}
{\sc Perrin-Riou, B.}
\newblock Groupe de {S}elmer d'une courbe elliptique \`a multiplication
  complexe.
\newblock {\em Compositio Math. 43}, 3 (1981), 387--417.

\bibitem{PR04}
{\sc Pollack, R., and Rubin, K.}
\newblock The main conjecture for {CM} elliptic curves at supersingular primes.
\newblock {\em Ann. of Math. (2) 159}, 1 (2004), 447--464.

\bibitem{Rib81}
{\sc Ribet, K.}
\newblock Torsion points of abelian varieties in cyclotomic extensions.
\newblock {\em Enseign. Math 27\/} (1981), 315--319.

\bibitem{Rub_AWS}
{\sc Rubin, K.}
\newblock Elliptic curves with complex multiplication and the conjecture of
  {B}irch and {S}winnerton-{D}yer, 1999.
\newblock AWS notes.

\bibitem{Rub00}
{\sc Rubin, K.}
\newblock {\em Euler systems}, vol.~147 of {\em Annals of Mathematics Studies}.
\newblock Princeton University Press, Princeton, NJ, 2000.
\newblock Hermann Weyl Lectures. The Institute for Advanced Study.

\bibitem{Sch79}
{\sc Schneider, P.}
\newblock \"{U}ber gewisse {G}aloiscohomologiegruppen.
\newblock {\em Math. Z. 168}, 2 (1979), 181--205.

\bibitem{Ser97}
{\sc Serre, J.-P.}
\newblock {\em Abelian {$l$}-adic representations and elliptic curves}.
\newblock W. A. Benjamin, Inc., New York-Amsterdam, 1968.
\newblock McGill University lecture notes written with the collaboration of
  Willem Kuyk and John Labute.

\bibitem{Ser81}
{\sc Serre, J.-P.}
\newblock Quelques applications du th\'{e}or\`eme de densit\'{e} de
  {C}hebotarev.
\newblock {\em Inst. Hautes \'{E}tudes Sci. Publ. Math.}, 54 (1981), 323--401.

\bibitem{Ser_GalCoho}
{\sc Serre, J.-P.}
\newblock {\em Galois cohomology}.
\newblock Springer-Verlag, Berlin, 1997.
\newblock Translated from the French by Patrick Ion and revised by the author.

\bibitem{Sha08}
{\sc Sharifi, R.~T.}
\newblock On {G}alois groups of unramified pro-{$p$} extensions.
\newblock {\em Math. Ann. 342}, 2 (2008), 297--308.

\bibitem{She18}
{\sc Shekhar, S.}
\newblock Comparing the corank of fine {S}elmer group and {S}elmer group of
  elliptic curves.
\newblock {\em J. Ramanujan Math. Soc. 33}, 2 (2018), 205--217.

\bibitem{Tak20}
{\sc Takahashi, N.}
\newblock On {G}reenberg's generalized conjecture for imaginary quartic fields.
\newblock {\em Int. J. Number Theory 17}, 5 (2021), 1163--1173.

\bibitem{Ven02}
{\sc Venjakob, O.}
\newblock On the structure theory of the {I}wasawa algebra of a {$p$}-adic
  {L}ie group.
\newblock {\em J. Eur. Math. Soc. (JEMS) 4}, 3 (2002), 271--311.

\bibitem{Ven03}
{\sc Venjakob, O.}
\newblock A non-commutative {W}eierstrass preparation theorem and applications
  to {I}wasawa theory.
\newblock {\em J. Reine Angew. Math. 559\/} (2003), 153--191.
\newblock With an appendix by Denis Vogel.

\bibitem{Ven03_Compositio}
{\sc Venjakob, O.}
\newblock On the {I}wasawa theory of {$p$}-adic {L}ie extensions.
\newblock {\em Compositio Math. 138}, 1 (2003), 1--54.

\bibitem{Was97}
{\sc Washington, L.~C.}
\newblock {\em Introduction to cyclotomic fields}, 2~ed., vol.~83 of {\em
  Graduate Texts in Mathematics}.
\newblock Springer-Verlag, New York, 1997.

\bibitem{Wut07_fineSelmer}
{\sc Wuthrich, C.}
\newblock Iwasawa theory of the fine {S}elmer group.
\newblock {\em J. Algebraic Geom. 16}, 1 (2007), 83--108.

\end{thebibliography}
\end{document}